\documentclass{amsproc}
\usepackage{amsfonts,amssymb,amsmath,amsthm}
\usepackage{url}
\usepackage{enumerate}

\newtheorem{theorem}{Theorem}[section]
\newtheorem*{thm}{Theorem}
\newtheorem{lemma}[theorem]{Lemma}
\newtheorem{proposition}[theorem]{Proposition}
\newtheorem{corollary}[theorem]{Corollary}
\theoremstyle{definition}
\newtheorem{definition}[theorem]{Definition}
\theoremstyle{remark}

\newtheorem{example}[theorem]{Example}
\numberwithin{equation}{section}

\def\sqr#1#2{{\,\vcenter{\vbox{\hrule height.#2pt\hbox{\vrule width.#2pt
height#1pt \kern#1pt\vrule width.#2pt}\hrule height.#2pt}}\,}}
\def\bo{\sqr44\,}
\def\vp{\varphi}

\def\tp#1{\{#1\}}
\def\tpp#1#2#3{\{{#1}{#2}{#3}\}}
\def\tpl#1#2#3{\{\tp{#1}#2#3\}}
\def\tpr#1#2#3{\{#1#2\tp{#3}\}}
\newcommand{\IC}{\mathbb{C}}
\newcommand{\tpc}[3]{\{#1,#2,#3\}}

\newcommand{\tr}{\hbox{tr}\, }

\newcommand{\matr}[4]{
\left[\begin{array}{cc}
#1&#2\\
#3&#4
\end{array}
\right]}

\begin{document}
\title{Cohomology of Jordan triples via Lie algebras}
\author{Cho-Ho Chu} 
\date{\today}
\address{School of Mathematical Sciences\\
Queen Mary, University of London\\
London E1 4NS, UK}
\email{c.chu@qmul.ac.uk}
\author{Bernard Russo}
\address{Department of mathematics, University of California, Irvine, USA}
\email{brusso@uci.edu}


\begin{abstract} We develop a cohomology theory for
Jordan triples, including the infinite dimensional ones,
by means of the cohomology of TKK Lie algebras. This
enables us to apply Lie cohomological results to the setting of
Jordan triples. Some preliminary results for von Neumann algebras are obtained.

\end{abstract}
\maketitle
\tableofcontents
\section{Introduction}
\begin{quotation}
A veritable army of researchers took the theory of derivations of operator algebras to dizzying heights---producing a theory of cohomology of operator algebras as well as much information about automorphisms of operator algebras---{\it Richard Kadison \cite{Kadison00}}
\end{quotation}

In addition to associative algebras, cohomology groups are defined for Lie algebras and, to some extent, for Jordan algebras.  Since the structures of Jordan derivations and Lie derivations on von Neumann algebras are well understood,  and in view of the above quotation, isn't it time to study the higher dimensional non associative cohomology of a von Neumann algebra?  The present paper is motivated by this rhetorical question.

In this paper  we develop a cohomology theory for
Jordan triples, including the infinite dimensional ones,
by means of the cohomology of TKK Lie algebras. This
enables us to apply Lie cohomological results to the setting of
Jordan triples.  Several references, which will be mentioned below,  use Lie theory as a tool to study Jordan cohomology.

The outline of the paper is the following.
In the rest of this introduction, we give an overview of various cohomology theories, both classical and otherwise. (For a more detailed survey see \cite{russo}.)
In section~\ref{sect2}, the definitions of Jordan triple module and Lie algebra module, as well as the  Tits-Kantor-Koecher (TKK) construction are reviewed, basically following \cite{chu}. It is shown in Theorem~\ref{m} that a Jordan triple module gives rise to a Lie module for the corresponding TKK algebra. The proof of Theorem~\ref{m} is deferred to subsection~\ref{app:1029141}.

After reviewing the cohomology of Lie algebras (with or without an involution) in section~\ref{sec:harris}, two infinite families of cohomology groups are defined for a Jordan triple system $V$
in section~\ref{sect4}, one using the Lie cohomology of the TKK algebra of $V$ and the other using the Lie cohomology of the TKK algebra with its canonical involution $\theta$.  A complete analysis is given for the first cohomology groups in Proposition~\ref{prop:1015141}, which shows that structural transformations on $V$ correspond to derivations of the TKK Lie algebra, and triple derivations on $V$ correspond to the  $\theta$-invariant derivations.

Section~\ref{sect5} contains examples of Jordan cocycles and TKK algebras, and applications, including a characterization of certain 3-cocycles in Theorem~\ref{prop:0928141}, the proof of which appears in  subsection~\ref{app:1029143}.  The applications to von Neumann algebras appear in Theorem~\ref{thm:5.7} and Corollary~\ref{cor:1101141}.

\subsection{Brief survey of cohomology theories}\label{sect:two}
The starting point for the cohomology theory of associative algebras is the paper of Hochschild from 1945 \cite{Hochschild45}. The standard reference of the theory is \cite{CarEil56}.  Two other useful references are due to Weibel (\cite{Weibel94},\cite{Weibel99}).

Shortly after the introduction of  cohomology for associative algebras, there appeared in \cite{CheEil48} a corresponding theory for Lie algebras.   We follow \cite{Jacobson62} for the definitions and initial results. Applications can be found in \cite{Fuks86} and \cite{Knapp88}.

The cohomology theory for Jordan algebras is less well developed than for associative and Lie algebras.  A starting point would seem to be the papers of Gerstenhaber in 1964 \cite{Gerstenhaber64} and Glassman in 1970 \cite{Glassman70PJM}, which concern arbitrary
nonassociative algebras.  A study focussed primarily on Jordan algebras is \cite{Glassman70JAlg}.

We next recall two fundamental results, namely, the Jordan analogs of the first and second Whitehead lemmas as described in
\cite{Jacobson57}.

\begin{theorem}[Jordan analog of first Whitehead lemma \cite{Jacobson51}]
Let $J$ be a finite dimensional semisimple Jordan algebra over a field of characteristic 0 and let $M$ be a $J$-module.  Let $f$ be a linear mapping of $J$ into $M$ such that
\[
f(ab)=f(a)b+af(b).
\]
Then there exist $v_i\in M, b_i\in J$ such that
\[
f(a)=\sum_i((v_ia)b-v_i(ab_i)).
\]
\end{theorem}

\begin{theorem}[Jordan analog of second Whitehead lemma \cite{Penico51}]\label{5.6}
Let $J$ be a finite dimensional  separable\footnote{Separable, in this context, means that the algebra remains semisimple with respect to all extensions of the ground field.  For algebraically closed fields, this is the same as being semisimple} Jordan algebra  and let $M$ be a $J$-module.  Let $f$ be a bilinear mapping of $J\times J$ into $M$ such that
\[
f(a,b)=f(b,a)
\]
and
\[
f(a^2,ab)+f(a,b)a^2+f(a,a)ab
=f(a^2b,a)+f(a^2,b)a+(f(a,a)b)a
\]
Then there exist a linear mapping $g$ from $J$ into $M$ such that
\[
f(a,b)=g(ab)-g(b)a-g(a)b
\]
\end{theorem}

Two proofs of Theorem~\ref{5.6} are given in \cite{Jacobson68}.  One of them, which uses the classification of finite dimensional Jordan algebras, is outlined in \cite[4.3.1]{russo}.   The other proof  uses Lie algebras and is contained in \cite[pp. 324--336]{Jacobson68}.

A study of low dimensional cohomology for quadratic Jordan algebras is given in \cite{McCrimmon71}.  Since quadratic Jordan algebras (which coincide with ``linear'' Jordan algebras over characteristic 0 fields) can be considered a bridge from Jordan algebras to Jordan triple systems, this would seem to be a good place to look for exploring cohomology theory for Jordan triples.   Indeed, this is hinted at in \cite{McCrimmon82}, since although  \cite{McCrimmon71} is about Jordan algebras, the concepts are phrased in terms of the associated triple product $\{abc\}=(ab)c+(cb)a-(ac)b$.

 However, both papers stop short of defining higher dimensional cohomology groups.
 The paper \cite{McCrimmon71}, which is mostly concerned with representation theory, proves,  for the only cohomology groups defined, the linearity of the functor $H^n$:
  \[
 H^n(J,\oplus_i M_i)=\oplus_i H^n(J,M_i),\quad n=1,2.
 \]
The paper \cite{McCrimmon82}, which is mostly concerned with compatibility of tripotents in Jordan triple systems, proves versions of the linearity of the functor $H^n$, $n=1,2$,  corresponding to the Jordan triple structure.

The earliest work on cohomology of triple systems seems to be \cite{Harris61} (Lie triple systems),  which is discussed in section~\ref{sec:harris}. Four decades later, the second paper on the cohomology of Lie triple systems appeared \cite{HodPar02}.

The following is from the review \cite{Seibtrev} of \cite{Carlsson76} (associative triple systems).
\begin{quotation}
``A cohomology for associative triple systems is defined, with the main purpose to get quickly the cohomological triviality of finite-dimensional separable objects over fields of characteristic $\ne  2$, i.e., in particular the Whitehead lemmas and the Wedderburn principal theorem.''
\end{quotation}

The authors of the present paper know of only two other references dealing with the Wedderburn principal theorem in the context of triple systems, namely, \cite{Carlsson77} (alternative triple systems) and \cite{KuhRos78} (Jordan triple systems).
In the latter paper, the  well-known Koecher-Tits-construction of a Lie algebra from a Jordan algebra is generalized to Jordan pairs.  The radical of this Lie algebra is calculated in terms of the given Jordan pair and a Wedderburn decomposition theorem for Jordan pairs (and triples) in the characteristic zero case is proved.

Finally, we mention that a more general approach to cohomology of algebras and triple systems appears in the paper of Seibt \cite{Seibt75}.

\section{Jordan triples and TKK Lie algebras} \label{sect2}

By a {\it Jordan triple}, we mean a real or complex vector space $V$, equipped
with a Jordan triple product $\{\cdot,\cdot,\cdot\} : V^3 \rightarrow V$
which is linear and symmetric in the outer variables,  conjugate linear in the middle variable, and
satisfies the Jordan triple identity
\[\{x,y,\{a,b,c\}\}=\{\{x,y,a\},b,c\}-\{ a,\{y,x,b\},c\}
+\{a,b,\{x,y,c\}\}\]
for $a,b,c,x,y \in V$. Given two elements $a, b$ in a Jordan triple $V$, we define
the {\it box operator} $a\bo b : V \rightarrow V$ by $a\bo b(\cdot) = \{a,b,\cdot\}$.

All Lie algebras in this paper are real or complex. We construct a
cohomology theory of Jordan triples using the Tits-Kantor-Koecher
(TKK) Lie algebras associated with them. Although we could develop
the theory for {\it all} Jordan triples, we focus on the
nondegenerate ones, which will be assumed throughout, to avoid
unnecessary complication. For degenerate Jordan triples, the
construction is exactly the same albeit more computation is
involved. A Jordan triple is called {\it nondegenerate} if for each $a\in V$, the condition
$\{a,a,a\}= 0$ implies $a=0$. Given that $V$ is
nondegenerate, one has
\[ \sum_j a_j \bo b_j = \sum_k c_k \bo d_k ~\Rightarrow~  \sum_k b_j \bo a_j = \sum_j d_k \bo c_k
\quad (a_j, b_j, c_k, d_k \in V)\]
which facilitates a simple definition of the TKK Lie algebra $\frak L (V)$ of $V$, with an invoultion $\theta$
(cf. \cite[p.45]{chu}), where
$$\frak L (V) = V \oplus V_0 \oplus V,$$
$V_0 = \{ \sum_j a_j \bo b_j: a_j, b_j \in V\}$,  the Lie product is defined by
\begin{equation}\label{eq:0418152}
[(x,h,y), (u,k,v)] = (hu - kx,\, [h,k]+x \bo v -u\bo y,\,
k^\natural y - h^\natural v),
\end{equation}
and for each $h= \sum_i a_i \bo b_i$ in the Lie subalgebra $V_0$ of
$\frak L (V)$, the map
$h^\natural : V \rightarrow V$  is well defined by
$$h^\natural = \sum_i b_i \bo a_i.$$
The involution  $\theta :
\frak L(V) \rightarrow \frak L(V)$ is given by
$$\theta (x,h,y) = (y,-h^\natural ,x) \qquad ((x,h,y) \in \frak L(V)).$$
Identifying $V$ with the subspace $\{(x,0,0) :x\in V\}$ of $\frak L(V)$, we have the
following relationship between the triple and Lie products:
$$\{a,b,c\} = [\,[a,\theta(b)],\, c] \qquad (a,b,c \in V).$$
If no confusion is likely, we often simplify the notation $\{a,b,c\}$ to $\{abc\}$.

Given a Lie algebra $\frak L$ and a module $X$ over $\frak L$, we
denote the action of $\frak L$ on $X$ by
$$(\ell, x) \in \frak L \times X \mapsto \ell . x \in X$$
so that
$$ [\ell, \ell '].x = \ell'.(\ell. x) - \ell .(\ell'. x).$$

\begin{definition}\label{mo} Let $V$ be a Jordan triple. A vector space $M$ over the same scalar field
is called a {\it Jordan triple $V$-module} (cf. \cite{russo}) if
it is equipped with three mappings
\[ \{\cdot,\cdot,\cdot\}_1: M \times V\times V \rightarrow M, \quad
\{\cdot,\cdot,\cdot\}_2: V \times M\times V \rightarrow M,
\{\cdot,\cdot,\cdot\}_3: V \times V\times M \rightarrow M\] such
that
\begin{enumerate}
\item [(i)] $\{a,b,c\}_1 = \{c,b,a\}_3$; \item[(ii)] $\{\cdot,
\cdot,\cdot\}_1$ is linear in the first two variables and
conjugate linear in the last variable, $\{\cdot,\cdot,\cdot\}_2$
is conjugate linear in all variables; \item[(iii)] denoting by $\{\cdot,\cdot,\cdot\}$
any of the products $\{\cdot,\cdot,\cdot\}_j ~(j=1,2,3)$, the
identity
$$\{a,b,\{c,d,e\}\} = \{\{a,b,c\},d,e\} -\{c,\{b,a,d\},e\} +
\{c,d,\{a,b,e\}\}$$ is satisfied whenever one of
the above elements is in $M$ and the rest in $V$.
\end{enumerate}
\end{definition}
For convenience, we shall omit the subscript $j$ from
$\{\cdot,\cdot,\cdot\}_j$ in the sequel. A $V$-module $M$ is
called {\it nondegenerate} if for each $m \in M$, each one of the conditions
$$ \{m,V,V\} =\{0\};~ \{V,m,V\}=\{0\}$$ implies $m=0.$
A nondegenerate Jordan triple $V$ is a nondegenerate module over
itself. For a JB*-triple $V$, its dual $V^*$ is a nondegenerate
$V$-module. {\it All Jordan triple modules throughout the paper
are assumed to be nondegenerate}.

  Given $a,b\in V$, the box operator
$a\bo b:V\rightarrow V$ can also  be considered as a mapping from
$M$ to $M$. Similarly, for $u\in V$ and $m \in M$, the ``{\it box
operators}''
$$u \bo m,\, m\bo u : V \longrightarrow M$$ are defined in a
natural way as $v\mapsto \{u,m,v\}$ and $v\mapsto \{m,u,v\}$
respectively.

Given $a,b\in V$, the identity (iii) in Definition \ref{mo}
implies
$$[a\bo b, u\bo m] = \{a,b,u\}\bo m - u\bo \{m,a,b\}$$
and
$$[a\bo b, m\bo u] = \{a,b,m\}\bo u - m\bo \{u,a,b\}.$$
for $u \in V$ and $m\in M$. We also have $[u\bo m, a\bo b] = \{u,m,a\}\bo b - a\bo\{b,u,m\}$
and similar identity for $[m\bo u, a\bo b]$.

Using similar arguments to the proof in \cite[Lemma
1.3.7]{chu}, one can show that
\begin{eqnarray}\label{sharp}
&& \sum_i u_i \bo m_i + \sum_j
n_j \bo v_j  = \sum_k u_k' \bo m_k' + \sum_\ell
n_\ell' \bo v_\ell'\\&\Rightarrow&   \sum_i m_i\bo u_i  + \sum_j
v_j  \bo n_j = \sum_k  m_k'\bo u_k' + \sum_\ell
v_\ell' \bo n_\ell'\nonumber
 \end{eqnarray} for $u_i, v_j,u_k', v_\ell' \in
V$ and $m_i, n_j m_k', n_\ell'\in M$.

  Let $M_0$ be the linear span of
$$\{u\bo m, n\bo v: u,v \in V, m,n \in M\}$$
in the vector space $L(V,M)$ of linear maps from $V$ to $M$. Then $M_0$ is  the space of {\it
inner structural transformations} $\hbox{Instrl}\, (V,M)$ (see \cite[Section 7]{McCrimmon82}) .
Extending the above product by linearity, we can define an action
of $V_0$ on $M_0$ by
$$ (h, \varphi) \in V_0 \times M_0 \mapsto  [h,\varphi]\in M_0.$$

\begin{lemma} $M_0$ is a $V_0$-module of the Lie algebra $V_0$.
\end{lemma}
\begin{proof}
We are required to show that
\begin{equation}\label{eq:0710141}
[[h,k],\varphi]=[h,[k,\varphi]]-[k,[h,\varphi]].
\end{equation}
We can assume that $h=a\bo b$, $k=c\bo d$ and $\varphi=w\bo m$ or $m\bo w$.  We assume $\varphi=w\bo m$, the other case being similar.
For the left side of (\ref{eq:0710141}), we have
\begin{eqnarray*}
[[a\bo b,u\bo v],w\bo m]&=&[\tp{abu}\bo v-u\bo\tp{vab},w\bo m]\\
&=&\tpp{\tp{abu}}{v}{w}\bo m-w\bo \tpp{m}{\tp{abu}}{v}\\
&-&\tpp{u}{\tp{vab}}{w}\bo m+w\bo \tpp{m}{u}{\tp{vab}}\\
&=&(\tpp{\tp{abu}}{v}{w}-\tpp{u}{\tp{vab}}{w})\bo m\\
&-&w\bo (\tpp{m}{\tp{abu}}{v}-\tpp{m}{u}{\tp{vab}}).
\end{eqnarray*}
For the right side of (\ref{eq:0710141}), we have
\begin{eqnarray*}
\lefteqn{[a\bo b,[u\bo v,w\bo m]]-[u\bo v,[a\bo b,w\bo m]]=}\\
&&[a\bo b,\tp{uvw}\bo m-w\bo \tp{muv}]-[u\bo v,\tp{abw}\bo m-w\bo\tp{mab}]\\
&=&\tpp{a}{b}{\tp{uvw}}\bo w-\tp{uvw}\bo\tp{mab}-\tp{abw}\bo\tp{muv}+w\bo\tpp{\tp{muv}}{a}{b}\\
&-&\tpp{u}{v}{\tp{abw}}\bo m+\tp{abw}\bo \tp{muv}+\tp{uvw}\bo\tp{mab}-w\bo\tpp{\tp{mab}}{u}{v}\\
&=&(\tpp{a}{b}{\tp{uvw}}-\tpp{u}{v}{\tp{abw}} )\bo m-w\bo(\tpp{\tp{mab}}{u}{v}-\tpp{\tp{muv}}{a}{b}).
\end{eqnarray*}
(\ref{eq:0710141}) now follows from the main identity for Jordan triples.
\end{proof}

Let $V$ be a Jordan triple and $\frak L (V)$ its TKK Lie algebra.
Given a triple $V$-module $M$, we now construct a corresponding
Lie module $\mathcal{L} (M)$ of the Lie algebra $\frak L (V)$ as
follows.

Let $$\mathcal{L} (M) = M \oplus M_0 \oplus M$$ and define the
action
$$ ((a,h,b), (m,\varphi,n)) \in \frak L(V) \times \mathcal{L} (M)
\mapsto (a,h,b). (m,\varphi,n) \in \mathcal{L} (M)$$ by
\begin{equation}\label{eq:0418151}
(a,h,b). (m,\varphi,n) = (hm-\varphi a,~ [h,\varphi] +a\bo
n - m\bo b,~ \varphi^\natural b - h^\natural(n)\,),
\end{equation}
 where, for $h
= \sum_i a_i\bo b_i$ and $\varphi = \sum_i u_i \bo m_i + \sum_j
n_j \bo v_j$, we have the following natural definitions
$$hm= \sum_i \{a_i,b_i,m\},\quad \varphi a= \sum_i \{u_i,  m_i,a\} +
\sum_j\{n_j,
 v_j, a\}, \quad  \varphi^\natural = \sum_i m_i \bo u_i +\sum_j v_j \bo
 n_j$$ in which $\vp^\natural$ is well-defined by (\ref{sharp}).

 \begin{theorem}\label{m} Let $V$ be a Jordan triple and let $\frak L (V)$
 be its TKK Lie algebra. Let $M$ be a triple $V$-module. Then
 $\mathcal{L} (M)$ is a Lie $\frak L(V)$-module.
 \end{theorem}

The proof of Theorem~\ref{m} consists of straightforward but tedious calculations.  Details can be found in subsection~\ref{app:1029141}.

\section{Cohomology of Lie algebras with involution}\label{sec:harris}

Let $T$ be a Lie triple system. Harris \cite[p.\,155]{Harris61} has developed  a
cohomology theory for $T$ in which the cohomology groups are derived from the ones
of its  enveloping Lie algebra
$\frak{L}_u = T + [T,
T]$ where $\frak{L}_u$  is equipped with an involution $\theta$ and the cochains in the cohomology complex are
  invariant under $\theta$.

Our Jordan triple cohomology makes use of TKK Lie algebras which are involutive. To pave
the way, we review briefly the cohomology for Lie algebras, with or without an involution.
Let $\frak{L}$ be a (real or complex) Lie algebra with involution
$\theta$.

\begin{definition} Given an involutive Lie algebra $(\frak{L},
\theta)$, an {\it $(\frak{L}, \theta)$-module} is a (left)
$\frak{L}$-module $\frak M$, equipped with an involution
$\widetilde\theta: \frak M \rightarrow \frak M$ satisfying
$$\widetilde\theta(\ell .\mu) = \theta(\ell). \widetilde\theta(\mu) \qquad
(\ell \in \frak{L}, \mu\in \frak M).$$ We also call $\frak M$ an
{\it involutive $\frak{L}$-module} if $\theta$ is understood.
\end{definition}
For $\ell \in \frak L$ and $\mu \in \frak M$, we define
$$[\ell, \mu]:= \ell.\mu \quad {\rm and} \quad [\mu,\ell]:= - \ell.\mu.$$

Let $\frak{L}^k=\overbrace{\frak{L} \times \cdots \times
\frak{L}}^{k-{\rm times}}$ be the $k$-fold cartesian product of
$\frak L$. A $k$-linear map $\psi: \frak {L}^k \rightarrow \frak
M$ is called {\it $\theta$-invariant} if
$$\psi (\theta x_1, \cdots, \theta x_k) = \widetilde \theta
\psi(x_1, \cdots, x_k) \quad {\rm for} \quad (x_1, \cdots, x_k)
\in \frak{L} \times \cdots \times \frak{L}.$$

Let $(\frak{L}, \theta)$ be an involutive Lie algebra and $\frak
M$ an $(\frak{L}, \theta)$-module. We define $A^0(\frak{L},\frak{M}) = \frak M$ and
$A_\theta^0(\frak{L},\frak{M})$ to be the 1-eigenspace of $\widetilde\theta$: $A_\theta^0(\frak{L}, \frak M)=\{\mu\in \frak{M}:\widetilde\theta\mu=\mu\}$.

For $k =1,2, \ldots$, we let
$$A^k(\frak{L}, \frak M) = \{\psi: \frak {L}^k \rightarrow \frak
M~ |~ \psi~ {\rm is~ k\hbox{-}linear ~ and~ ~ alternating}\}\quad
{\rm and}$$
$$A^k_\theta(\frak{L}, \frak M) = \{\psi \in A^k(\frak{L}, \frak M):
 |~ \psi~ {\rm is~ \theta\hbox{-}invariant}\}.$$

For $k=0,1,2,\ldots$, we define the {\it coboundary operator} $d_k
: A^k(\frak{L}, \frak M) \rightarrow
A^{k+1}(\frak{L}, \frak M)$ by $d_0 m(x)= x.m$ and for $k \geq 1$,
\begin{eqnarray}\label{eq:1015141}
\nonumber
(d_k\psi)(x_1, \ldots, x_{k+1})& =&  \sum_{\ell =1}^{k+1}
(-1)^{\ell
+1} x_\ell . \psi(x_1, \ldots, \widehat x_\ell, \ldots, x_{k+1})\\
&& + \sum_{1\leq i<j\leq k+1} (-1)^{i+j} \psi([x_i,x_j], \ldots,
\widehat x_i, \ldots, \widehat x_j, \ldots, x_{k+1})
\end{eqnarray}
where the symbol $\widehat z$ indicates the omission of $z$. The restriction of
 $d_k$ to the subspace $A^k_\theta(\frak{L}, \frak M)$, still denoted by $d_k$,
  has range $A^k_\theta(\frak{L}, \frak M)$  since a simple verification shows
that $d_k\psi$ is $\theta$-invariant and alternating  whenever
$\psi$ is. Also, we have $d_k d_{k-1}=0$ for $k=1,2, \ldots$ (cf.
\cite[p.\,167]{Knapp88}) and the two cochain complexes
$$A^0(\frak{L},
\frak M) \quad \longrightarrow ^{\mbox{\hspace{-.2in} $^{d_0}$}}
\quad A^1(\frak{L}, \frak M) \quad \longrightarrow
^{\mbox{\hspace{-.2in} $^{d_1}$}}\quad A^2(\frak{L}, \frak M)\quad
\longrightarrow ^{\mbox{\hspace{-.2in} $^{d_2}$}}\quad \cdots $$
$$A_\theta^0(\frak{L},
\frak M) \quad \longrightarrow ^{\mbox{\hspace{-.2in} $^{d_0}$}}
\quad A_\theta^1(\frak{L}, \frak M) \quad \longrightarrow
^{\mbox{\hspace{-.2in} $^{d_1}$}}\quad A_\theta^2(\frak{L}, \frak
M)\quad \longrightarrow ^{\mbox{\hspace{-.2in} $^{d_2}$}}\quad
\cdots .$$ We often omit the subscript $k$ from $d_k$ if there is
no ambiguity.

As usual, we define the {\it $k$-th cohomology group of $\frak L
$ with coefficients in
$\frak M$} to be the quotient
$$H^k(\frak{L},
\frak{M}) = \ker d_k /d_{k-1}(A^{k-1}(\frak{L}, \frak{M}))= \ker d_k /\hbox{im}\, d_{k-1}$$
for $k=1,2,\ldots $ and  define $H^0(\frak L, \frak{M}) =
\ker d_0$.
We  define the {\it $k$-th involutive cohomology group of $(\frak L,
\theta)$ with coefficients in an $(\frak L, \theta)$-module}
$\frak M$ to be the quotient
$$H_\theta^k(\frak{L},
\frak{M}) = \ker d_k /d_{k-1}(A^{k-1}_\theta(\frak{L}, \frak{M}))= \ker d_k /\hbox{im}\, d_{k-1}$$
for $k=1,2,\ldots $ and  define $H_\theta^0(\frak L, \frak{M}) =
\ker d_0 \subset H^0(\frak{L},
\frak{M})$.

For $k=1,2, \ldots$,  the map
$$\psi + d_{k-1}(A^{k-1}_\theta(\frak{L}, \frak{M})) \in H_\theta^k(\frak{L},
\frak{M})  \mapsto \psi + d_{k-1}(A^{k-1}(\frak{L}, \frak{M})) \in H^k(\frak{L},
\frak{M})$$ identifies $H_\theta^k(\frak{L},
\frak{M})$ as a subgroup of $H^k(\frak{L},
\frak{M})$.

\section{Cohomology of Jordan triples}\label{sect4}

\subsection{The cohomology groups}

Let $V$ be a Jordan triple and let $\frak{L}(V)= V \oplus V_0
\oplus V$ be its TKK Lie algebra with the  involution
$\theta(a,h,b)=(b,-h^\natural,a)$. Given a $V$-module $M$, we have
shown in Theorem \ref{m} that $\frak{L}(M)= M \oplus M_0\oplus M$
is an $\frak{L}(V)$-module. We define an induced involution
$\widetilde\theta :\frak{L}(M) \rightarrow \frak{L}(M)$ by
$$\widetilde \theta(m,\vp,n) = (n, -\vp^\natural, m)$$ for
$(m,\vp,n) \in M \oplus M_0\oplus M.$\\

\begin{lemma}$\frak{L}(M)$ is an $(\frak{L}(V), \theta)$-module,
that is,  we have $\widetilde\theta(\ell .\mu) = \theta(\ell).
\widetilde\theta(\mu)$ for $\ell \in \frak{L}(V)$ and $\mu \in
\frak{L}(M)$.
\end{lemma}

Let $\frak{k}(V) = \{(v,h,v)\in \frak{L}(V): h^\natural=-h\}$ be
the $1$-eigenspace of the involution $\theta$ (see
\cite[p.48]{chu}), which is a real Lie subalgebra of
$\frak{L}(V)$, and let $\frak{k}(M) = \{(m, \vp,m)\in \frak{L}(M):
\vp = -\vp^\natural\}$ be the $1$-eigenspace of $\widetilde
\theta$. Then $\frak{k}(M)$ is a Lie module over the Lie algebra
$\frak{k}(V)$. We will construct cohomology groups of a Jordan
triple $V$ with coefficients in a $V$-module $M$ using the
cohomology groups of $\frak{L}(V)$ with coefficients
$\frak{L}(M)$. For a real Jordan triple $V$, one can also make use of
the cohomology groups of the real Lie algebra
$\frak{k}(V)$ with coefficients $\frak{k}(M)$.\\

Let  $V$ be a Jordan triple. As usual, $V$ is identified as the
subspace
$$\{(v,0,0): v \in V\}$$
of the TKK Lie algebra $\frak{L}(V).$
For a triple $V$-module $M$, there is a natural embedding of $M$
into $\frak{L}(M) = M \oplus M_0 \oplus M$ given by
$$\iota: m\in M \mapsto  (m,0,0) \in \frak{L}(M)$$ and we will
identify $M$ with $\iota(M)$. We denote by $\iota_p :
\frak{L}(M)\rightarrow \iota(M)$ the natural projection
$$\iota_p(m, \vp, n) = (m,0,0).$$

We define $A^0(V,M) = M$ and for $k=1,2,\ldots$, we denote by
$A^k(V,M)$ the vector space of all alternating $k$-linear maps
$\omega:V^k = \overbrace{V \times \cdots \times V}^{k-{\rm times}}
\rightarrow M$.

Given $m \in M$, we define
$$\frak {L}_0(m) = (m, 0, 0) \in \frak{L}(M)$$
and view $\frak{L}_0(m)$ as an extension of $m\in A^0(V,M)$ to an
element in $A^0(\frak{L}(V), \frak{L}(M)) = \frak{L}(M)$.

To motivate the  definition of an extension $\frak{L}_k(\omega)\in
A^k(\frak{L}(V), \frak{L}(M))$ of an element $\omega \in
A^k(V,M)$, for $k\ge 1$, we first consider the case  $k=1$ and
note that $\omega \in A^1(V,M)$  is a Jordan triple derivation
 if
and only if
$$\omega \circ(a\bo b) - (a \bo b)\circ \omega =\omega(a)\bo b+a\bo \omega (b).$$ Let us call
a linear transformation $\omega:V\rightarrow M$ {\it extendable}
if the following condition holds:
$$
\sum_i a_i\bo b_i=0\Rightarrow \sum_i(\omega(a_i)\bo b_i+a_i\bo \omega(b_i))=0.
$$
Thus a Jordan triple derivation is extendable, and if $\omega$ is
any extendable transformation in $A^1(V,M)$, then the map
\[
\frak{L}_1(\omega) (x_1 \oplus a_1\bo b_1 \oplus y_1)
:= (\, \omega(x_1),\,  \omega(a_1)\bo b_1 + a_1\bo  \omega(b_1),\,\,\omega
(y_1)\,)
\]
is well defined and extends linearly to an element
$\frak{L}_1(\omega)\in A^1(\frak{L}(V),\frak{L}(M))$, in which
case we call $\frak{L}_1(\omega)$ the {\it Lie extension} of
$\omega$ on the Lie algebra $\frak{L}(V)$.

Now for $k>1$,  given a $k$-linear mapping $\omega : V^k
\rightarrow M$, we say that $\omega$ is {\it extendable} if it
satisfies the following condition under the assumption $\sum_i
u_i\bo v_i=0$:
$$
\sum_i\left(\omega(u_i,a_2,\ldots,a_k)\bo (v_i+b_2+\cdots +b_k)+
(u_i+a_2+\cdots+a_k)\bo \omega(v_i,b_2,\ldots, b_k)\right)=0,
$$
for all $a_2,\ldots,a_k,b_2,\ldots, b_k\in V$.

For an extendable $\omega$,  we can unambiguously define a
$k$-linear map  $\frak{L}_k (\omega): \frak{L}(V)^k \rightarrow
\frak{L}(M)$ as the linear extension (in each variable) of
\begin{eqnarray}\label{eq:1015142}
&&\frak{L}_k(\omega) (x_1 \oplus a_1\bo b_1 \oplus y_1,\, x_2
\oplus
a_2\bo b_2 \oplus y_2, \,\cdots, \,x_k \oplus a_k\bo b_k \oplus y_k)\\
&=& (\, \omega(x_1, \ldots, x_k),\,\, \sum_{j=1}^k \omega(a_1,
\ldots, a_k)\bo b_j +  \sum_{j=1}^k a_j\bo  \omega(b_1, \ldots, b_k),\,\,\omega
(y_1, \ldots, y_k)\,).\nonumber
\end{eqnarray}
 We call $\frak{L}_k(\omega)$ the {\it Lie extension} of $\omega$
 and often omit the subscript $k$ if no
confusion is likely. The following lemma is easy to verify.

\begin{lemma} Given an extendable $\omega \in A^{k}(V,M)$, we have $\frak{L}_k(\omega)\in
A^k(\frak{L}(V), \frak{L}(M))$. Moreover,  $\frak{L}_k(\omega)\in
A^k_\theta(\frak{L}(V), \frak{L}(M))$ if and only if  $k$ is odd.
\end{lemma}

This lemma enables us to define the following extension map on the
subspace $A^k(V,M)'$ of extendable maps in $A^k(V,M)$:
$$\frak{L}_k:
\omega \in A^{k}(V,M)' \mapsto \frak{L}_k(\omega) \in
A^k(\frak{L}(V), \frak{L}(M)).$$

Conversely, given $\psi \in A^{k}(\frak{L}(V), \frak{L}(M))$ for
 $k=1,2, \ldots$, one can define an alternating  map
$$J^k(\psi): V^{k} \rightarrow M$$ by
$$J^k(\psi)(x_1, \ldots,  x_{k}) = \iota_p
 \psi(\,(x_1,0,0),\, \ldots\,,   (x_{k}, 0, 0)\,)
$$
for $(x_1, \ldots,x_k) \in V^k$. We define $J^0: \frak{L}(M) \rightarrow \iota(M) \approx M
= A^0(V,M)$ by
$$J^0 (m, \vp, n) = (m,0,0) \qquad ((m,\vp,n) \in \frak{L}(M)).$$
We call $J^k(\psi)$ the {\it Jordan restriction}  of $\psi$ in
$A^k(V,M)$ and sometimes write $J$ for $J^k$ if the index $k$ is
understood.

\begin{example}\label{example:inner} Given a map  $\psi \in A^{k}_\theta(\frak{L}(V), \frak{L}(M))$,
we need not have $\psi((x_1,0,0), \ldots, (x_k,0,0)) \in \iota (M)$. Consider the inner derivation
${\rm ad} (m,\vp,m) \in A^{1}_\theta(\frak{L}(V), \frak{L}(M))$ defined by
$${\rm ad} (m,\vp,m)(x \oplus a\bo b\oplus y) = (x \oplus a\bo b\oplus y).
(m,\vp,m).$$
For $x\in V$, we have
$$ {\rm ad}(m,\vp,m)(x,0,0) =(x,0,0)\cdot(m,\varphi,m)= (-\vp(x), x\bo m, 0)
\notin \iota(M).$$
\end{example}

With the identification of $M$ and $\iota(M)$, the map $$J^k: \psi
\in A^{k}(\frak{L}(V), \frak{L}(M)) \mapsto A^k(V, M)$$ can be
viewed as the left inverse of $\frak{L}_k : A^k(V,M)' \rightarrow
A^{k}(\frak{L}(V), \frak{L}(M))$ since for an extendable $\omega$,
we have
\begin{eqnarray*}
J^k\frak{L}_k(\omega)(x_1, \ldots, x_k) &=& \iota_p \frak{L}(\omega)(
(x_1,0,0),\, \ldots\,,   (x_{k}, 0, 0))\\
&=&\omega(x_1, \ldots, x_k).
\end{eqnarray*}

We can now define
the cohomology groups for a Jordan triple $V$ with coefficients $M$ by means
of the cochain complexes for the Lie algebra $\frak{L}(V)$ and the involutive
Lie algebra $(\frak{L}(V),\theta)$:

$$\begin{array}{cccccccc} \frak{L}(M)= A^0(\frak{L}(V), \frak{L}(M)) &
\longrightarrow ^{\mbox{\hspace{-.2in} $^{d_0}$}} &
 A^1(\frak{L}(V), \frak{L}(M)) & \longrightarrow
^{\mbox{\hspace{-.2in} $^{d_1}$}} &
A^2(\frak{L}(V), \frak{L}(M)) & \longrightarrow
^{\mbox{\hspace{-.2in} $^{d_2}$}}& \cdots\\\
\\
\downarrow {J^0}
 &
&\downarrow {J^1}&&
\downarrow {J^2}&&\cdots&
\\\\
M= A^0(V,M)& & A^1(V,M) &&
A^2(V,M)&& \cdots &
\end{array}
$$
$$ $$
$$\begin{array}{cccccccc}  A_\theta^0(\frak{L}(V), \frak{L}(M)) &
\longrightarrow ^{\mbox{\hspace{-.2in} $^{d_0}$}} &
 A_\theta^1(\frak{L}(V), \frak{L}(M)) & \longrightarrow
^{\mbox{\hspace{-.2in} $^{d_1}$}} &
A_\theta^2(\frak{L}(V), \frak{L}(M)) & \longrightarrow
^{\mbox{\hspace{-.2in} $^{d_2}$}}& \cdots\\\
\\
\downarrow {J^0}
 &
&\downarrow {J^1}&&
\downarrow {J^2}&&\cdots&
\\\\
M= A^0(V,M)& & A^1(V,M) &&
A^2(V,M)&& \cdots &
\end{array}
$$
For $k=0,1,2, \ldots$, the {\it $k$-th cohomology groups} $H^k(V,M)$ are defined by
\begin{eqnarray*}
H^0(V,M) &=&  J^0(\ker d_0) = J^0\{(m,\vp,n): (u,h,v).(m,\vp,n)=0,\,\forall (u,h,v)\in \frak{L}(V)\}\\
&=& \{m\in M:  m \bo v =0, \, \forall v\in V\} = \{0\}
\end{eqnarray*} and
$$H^k(V,M) = Z^k(V,M)/B^{k}(V,M) \qquad (k=1,2,\ldots)$$
where $$Z^k(V,M) = J^k(Z^k(\frak{L}(V),\frak{L}(M))), \quad Z^k(\frak{L}(V),\frak{L}(M))
= \ker d_k$$
and
$$B^{k}(V,M) =  J^k(B^k(\frak{L}(V),\frak{L}(M))), \quad B^k(\frak{L}(V),\frak{L}(M))
={\rm im}\, d_{k-1}.$$

For $k=0,1,2 \ldots$, the {\it $k$-th involutive cohomology groups} $H^k_\theta(V,M)$ are defined by
$$H^0(V,M) = J^0(\ker d_0) = \{0\}$$ and
$$H^k_\theta(V,M) = Z^k_\theta(V,M)/B^{k}_\theta(V,M) \qquad (k=1,2,\ldots)$$
where $$Z^k_\theta(V,M) = J^k(Z^k_\theta(\frak{L}(V),\frak{L}(M))), \quad
Z^k_\theta(\frak{L}(V),\frak{L}(M))) = \ker d_k|_{A_\theta^k(\frak{L}(V), \frak{L}(M))}$$
and
$$B^{k}_\theta(V,M) =  J^k(B^k_\theta(\frak{L}(V),\frak{L}(M))), \quad
B^k_\theta(\frak{L}(V),\frak{L}(M)) = d_{k-1} (A_\theta^{k-1}(\frak{L}(V), \frak{L}(M))).$$

We see that the map $$\omega + B^{k}_\theta(V,M) \in H^k_\theta(V,M)
\mapsto \omega + B^{k} (V,M) \in H^k(V,M) $$
identifies $H^k_\theta(V,M) $ as a subgroup of $H^k(V,M) $. We call elements in $H^k(V,M)$ the {\it Jordan triple
$k$-cocycles}, and the ones in $H^k_\theta(V,M) $ the {\it involutive Jordan triple
$k$-cocycles}. Customarily, elements in $B^k(V,M)$ are called the {\it coboundaries}.

\subsection{Triple derivations}
\begin{definition} Let $V$ be a Jordan triple and $M$ a triple $V$-module. A mapping
$\omega : V \rightarrow M$ is called an {\it inner triple derivation} if it is
of the form
$$\omega = \sum_{i=1}^k (m_i \bo v_i - v_i \bo m_i) \in M_0$$
for some $m_1, \ldots, m_k \in M$
and $v_1, \ldots, v_k\in V$.
Note that $\omega^\natural=-\omega$ and $(0,\omega,0)\in\frak{k}(M)$.
\end{definition}

Let us compute the first involutive cohomology group $H^1_\theta(V,M) = Z^1_\theta(V,M)/B^{1}_\theta(V,M)$.  First,
we show that $B^{1}_\theta(V,M)$ coincides with the space of inner triple derivations from $V$ to $M$.

Let $\omega$ be an  inner triple drivation on $V$. We show that its Lie extension
$\frak{L}(\omega)$ is a Lie inner derivation on the Lie algebra
$\frak{L}(V)$. Indeed, we have
\begin{eqnarray*}
\frak{L}(\omega) (x \oplus a\bo b\oplus y) &=&
(\omega(x),\, \omega(a)\bo b+ a\bo \omega(b),\, \omega(y))\\
&=&(x \oplus a\bo b\oplus x). (0, -\omega,0).
\end{eqnarray*}
Hence $\omega = J^1(\frak{L}_1(\omega)) \in B^{1}_\theta(V,M)$,
where $(0,-\omega,0)\in\frak{k}(M)$. Conversely, let $\psi = {\rm
ad}(m, \vp, n) \in A_\theta^1(\frak{L}(V), \frak{L}(M))$ be a Lie
inner derivation. Then for $x\in V$, we have
\begin{eqnarray*}
J^1(\psi)(x)&=& \iota_p\psi(x,0,0) = \iota_p(x,0,0).(m, \vp,n)\\& =& \iota_p(-\vp(x), x\bo n,0)
=-\vp(x),
\end{eqnarray*}
where $\widetilde \theta (\vp) = \vp$ implies that $\vp : V \rightarrow M$ is an inner triple derivation.

We now show that $Z^1_\theta(V,M)$ coincides with the set of  triple derivations of $V$.

\begin{lemma}\label{lem:0922141}
Let $\omega : V \rightarrow M$ be a triple derivation. Then
$\frak{L}(\omega): \frak{L}(V) \rightarrow \frak{L}(M)$
is a $\theta$-invariant Lie derivation.
\end{lemma}
\begin{proof}
For notation's sake we denote $\frak{L}(\omega)$ by $D$.  Thus
$$
D(x,a\bo b,y)=(\omega(x),\omega(a)\bo b+a\bo\omega(b),\omega(y)),
$$
and it is clear that $D$ is $\theta$-invariant.
We need to verify
$$
D[(x,a\bo b,y),(u,c\bo d,v)]=(x,a\bo b,y)\cdot D(u,c\bo d,v)-(u,c\bo d,v)\cdot D(x,a\bo b,y).
$$
for $(x,a\bo b,y),(u,c\bo d,v)\in \frak{L}(V)$.
By writing
\begin{eqnarray*}
D[(x,a\bo b,y),(u,c\bo d,v)]&=&D[(x,0,y),(u,0,v)]+D[(0,a\bo b,0),(u,0,v)]\\
&+&D[(0,a\bo b,0),(0,c\bo d,0)]+D[(x,0,y),(0,c\bo d,0)],
\end{eqnarray*}
we only need to verify the three identities
\begin{equation}\label{eq:0917141}
D[(x,0,y),(u,0,v)]=(x,0,y)\cdot D(u,0,v)-(u,0,v)\cdot D(x,0,y),
\end{equation}
\begin{equation}\label{eq:0917142}
D[(0,a\bo b),(u,0,v)]=(0,a\bo b,0)\cdot D(u,0,v)-(u,0,v)\cdot D(0,a\bo b,0),
\end{equation}
and
\begin{equation}\label{eq:0917143}
D[(0,a\bo b,0),(0,c\bo d,0)]=(0,a\bo b,0)\cdot D(0,c\bo d,0)-(0,c\bo d,0)\cdot D(0,a\bo b,0).
\end{equation}
These are easy consequences of the definitions.  For completeness we include details.
The left side of  (\ref{eq:0917141}) is
\begin{eqnarray*}
D[(x,0,y),(u,0,v)]&=&D(0,x\bo v-u\bo y,0)\\
&=&(0,\omega(x)\bo v+x\bo \omega(v)-\omega(u)\bo y-u\bo\omega(y),0),
\end{eqnarray*}
and the right side is
\begin{eqnarray*}
\lefteqn{(x,0,y)\cdot D(u,0,v)-(u,0,v)\cdot D(x,0,y)=}\\
&&(x,0,y)\cdot(\omega(u),0,\omega(v))-(u,0,v)\cdot(\omega(x),0,\omega(y))\\
&=&(0,x\bo \omega(v)-\omega(u)\bo y,0)-(0,u\bo\omega(y)-\omega(x)\bo v,0),
\end{eqnarray*}
proving
(\ref{eq:0917141}).
The left side of  (\ref{eq:0917142}) is
\begin{eqnarray*}
D[(0,a\bo b,0),(u,0,v)]&=&D(\tp{abu},0,-\tp{bav})=(\omega\tp{abu},0,-\omega \tp{bav})\\
\end{eqnarray*}
and the right side is
\begin{eqnarray*}
\lefteqn{(0,a\bo b,0)\cdot D(u,0,v)-(u,0,v)\cdot D(0,a\bo b,0)=}\\
&&(0,a\bo b,0)\cdot(\omega(u),0,\omega(v))-(u,0,v)\cdot(0,\omega(a)\bo b+a\bo\omega(b),0)\\
&=&(\tp{ab\omega(u)},0,-\tp{ba\omega(v)})-(\tp{-\omega(a)bu}-\tp{a\omega(b)u},0,\tp{b\omega(a)v}+\tp{\omega(b)av}),
\end{eqnarray*}
proving
(\ref{eq:0917142}).
The left side of  (\ref{eq:0917143}) is
\begin{eqnarray*}
D[(0,a\bo b,0),(0,c\bo d,0)]&=&D(0,[a\bo b,c\bo d],0)=D(0,\tp{abc}\bo d-c\bo\tp{dab},0)\\
&=&(0,\omega\tp{abc}\bo d+\tp{abc}\bo \omega(d)-\omega(c)\bo\tp{dab}-c\bo\omega\tp{dab},0)\\
&=&(0,\tp{\omega(a)bc}\bo d+\tp{a\omega(b)c}\bo d+\tp{ab\omega(c)}\bo d+\tp{abc}\bo \omega(d)\\
&&-\omega(c)\bo\tp{dab}-c\bo\tp{\omega(d)ab}-c\bo \tp{d\omega(a)b}-c\bo\tp{da\omega(b)},0),
\end{eqnarray*}
and the right side is
\begin{eqnarray*}
\lefteqn{(0,a\bo b,0)\cdot D(0,c\bo d,0)-(0,c\bo d,0)\cdot D(0,a\bo b,0)}\\
&=&(0,a\bo b,0)\cdot(0,\omega(c)\bo d+c\bo\omega(d),0)-(0,c\bo d,0)\cdot(0,\omega(a)\bo b+a\bo\omega(b),0)\\
&=&(0,[a\bo b,\omega(c)\bo d]+[a\bo b,c\bo\omega(d)]-[c\bo d,\omega(a)\bo b]-[c\bo d,a\bo\omega(b)],0)\\
&=&(0,[a\bo b,\omega(c)\bo d]+[a\bo b,c\bo\omega(d)]+[\omega(a)\bo b,c\bo d]+[a\bo \omega(b),c\bo d],0\\
&=&(0,\tp{ab\omega(c)}\bo d-\omega(c)\bo\tp{dab}+\tp{abc}\bo\omega(d)-c\bo\tp{\omega(d)ab}\\
&&+\tp{\omega(a)bc}\bo d-c\bo \tp{d\omega(a)b}+\tp{a\omega(b)c}\bo d-c\bo\tp{da\omega(b)},0)
\end{eqnarray*}
proving
(\ref{eq:0917143}).\end{proof}

The previous lemma shows that all triple derivations $\omega$ on
$V$ are contained in $Z^1_\theta(V,M)$.   Conversely, given a Lie
derivation $\psi \in A_\theta^1(\frak{L}(V), \frak{L}(M))$, we
show below that $J(\psi)$ is a triple derivation on $V$. This
shows that every element in $Z_\theta^1(V,M)$ is a triple
derivation and hence $H^1_\theta(V,M)$ is  the space of triple
derivations modulo the inner triple derivations of $V$ into $M$.
This will be generalized in the next subsection.

\subsection{Structural Transformations}
A (conjugate-) linear transformation $S:V\rightarrow M$ is said to be a {\it
structural transformation} if there exists a (conjugate-) linear transformation
$S^*:V\rightarrow M$ such that
\[
S\tp{xyx}+\tp{x(S^*y)x}=\tp{xySx}
\]
and
\[
S^*\tp{xyx}+\tp{x(Sy)x}=\tp{xyS^*x}.
\]
A triple derivation $D$ is a special case of a structural transformation with $D^*=-D$.  By polarization,
this property is equivalent to
\[
S\tp{xyz}+\tp{x(S^*y)z}=\tp{zySx}+\tp{xySz}
\]
and
\[
S^*\tp{xyz}+\tp{x(Sy)z}=\tp{zyS^*x}+\tp{xyS^*z}.
\]

As noted earlier, the space of {\it inner structural
transformations} coincides, by definition, with the space $M_0$. Triple derivations which are inner structural transformations are inner triple derivations. Also,  if $\omega$ is a
structural transformation, then $\omega-\omega^*$ is a triple
derivation and if $\omega$ is a triple derivation, then $i\omega$ is a structural transformation which is inner if $\omega$ is inner.

\begin{proposition}\label{prop:1015141}
Let $\psi$ be a Lie
derivation of $\frak{L}(V)$ into $\frak{L}(M)$.  Then
\begin{enumerate}
\item[\rm(i)]
$J(\psi): V \rightarrow M$ is a structural transformation with $(J\psi)^*=-J\psi'$ where $\psi'=\widetilde\theta\psi\theta$.
\item[\rm(ii)]  If $\psi$ is $\theta$-invariant, then $\psi'=\psi$ and  $J\psi$ is a  triple derivation.
\item[\rm(iii)] If $\psi$ is an inner derivation then $J\psi$ is an inner structural transformation.
In particular, if $\psi$ is a $\theta$-invariant inner derivation then $J\psi$ is an inner triple derivation.
\end{enumerate}
Conversely, let $\omega$ be a structural transformation.
\begin{enumerate}
\item[\rm(iv)] The mapping $D=\frac{1}{2}\frak{L}_1(\omega-\omega^*):\frak{L}(V)\rightarrow \frak{L}(M)$ defined by
\[
D(x,a\bo b,y)=\frac{1}{2}(\omega(x)-\omega^*(x),\omega(a)\bo b-a\bo \omega^*(b)-\omega^*(a)\bo b+a\bo\omega(b),\omega(y)-\omega^*(y))
\]
is a derivation of the Lie algebra $\frak{L}(V)$ into $\frak{L}(M)$.
\item[\rm(v)]
$D$ is $\theta$-invariant if and only if $\omega$ is a triple derivation, that is, $\omega^*=-\omega$.
\item[\rm(vi)] If $\omega$ is an inner structural transformation then $D$ is an inner derivation.
In particular, if $\omega$ is an inner triple derivation then $D$ is a $\theta$-invariant  inner  derivation.
\end{enumerate}

\end{proposition}
\begin{proof} Let $\psi$ be a Lie
derivation of $\frak{L}(V)$ into $\frak{L}(M)$.
We show first that
\begin{equation}\label{eq:0919141}
J\psi\tp{abc}=\tp{abJ\psi(c)}+\tp{a,J\psi'(b),c}+\tp{J\psi(a)bc}.
\end{equation}
Let us define  $n:V\rightarrow  M$, and $n_1:V\rightarrow M$ by the formulas $\psi(0,0,x)=(m(x),\varphi(x),n(x))$, and
 $\psi(x,0,0)=(J\psi(x),\varphi_1(x),n_1(x))$.  Then
\begin{eqnarray*}
\lefteqn{(J\psi\tp{abc},\varphi_1\tp{abc},n_1\tp{abc})=\psi(\tp{abc},0,0)}\\
&=&\psi[[(a,0,0),(0,0,b)],(c,0,0)]\\
&=&[(a,0,0),(0,0,b)]\cdot\psi(c,0,0)-(c,0,0)\cdot\psi[(a,0,0),(0,0,b)]\\
&=&(0,a\bo b,0)\cdot(J\psi(c),\varphi_1(c),n_1(c))-(c,0,0)\cdot((a,0,0)\cdot\psi(0,0,b)-(0,0,b)\cdot\psi(a,0,0))\\
&=&(\tp{abJ\psi(c)},[a\bo b,\varphi_1(c)],-\tp{ban_1(c)})\\
&&-(c,0,0)\cdot((a,0,0)\cdot (m(b),\varphi(b),n(b))-(0,0,b)\cdot (J\psi(a),\varphi_1(a),n_1(a)))\\
&=&(\tp{abJ\psi(c)},[a\bo b,\varphi_1(c)],-\tp{ban_1(c)})\\
&&-(c,0,0)\cdot(-\varphi(b)(a),a\bo n(b),0)-(c,0,0)\cdot (0,-J\psi(a)\bo b,\varphi_1(a)^\natural b)\\
&=&(\tp{abJ\psi(c)},[a\bo b,\varphi_1(c)],-\tp{ban_1(c)})\\
&&-(-\tp{an(b)c},0,0))-(\tp{J\psi(a)bc},c\bo\varphi_1(a)^\natural
b,0).
\end{eqnarray*}

Note that
\begin{eqnarray*}
(m(b),\varphi(b),n(b))&=&\psi(0,0,b)\\&=&\widetilde\theta\psi'\theta(0,0,b)\\
&=&\widetilde\theta\psi'(b,0,0)\\
&=&\widetilde\theta(J\psi'(b),\varphi'(b),n'(b))\\
&=&(n'(b),-\varphi'^\natural (b),J\psi'(b)),
\end{eqnarray*}
so that $n(b)=J\psi'(b)$, proving (\ref{eq:0919141}).

Applying  (\ref{eq:0919141}) to  $\psi'=\widetilde\theta\psi\theta$, we have,
since $\psi''=\psi$
\begin{equation}\label{eq:0919142}
J\psi'\tp{abc}=\tp{abJ\psi'(c)}+\tp{a,J\psi(b),c}+\tp{J\psi'(a)bc}
\end{equation}
proving (i).

If $\psi$ is $\theta$-invariant, then $\psi'=\psi$
so that $J\psi$ is a triple derivation, proving (ii).
Example~\ref{example:inner} provides a proof of (iii).

(iv) is immediate from Lemma~\ref{lem:0922141} since $\omega-\omega^*$ is a triple
derivation.  The definitions show that $\widetilde\theta D\theta=D$ if and only if $\omega=(\omega-\omega^*)/2$, proving (v). Finally, if $\omega$ is an inner structural transformation, then $\omega-\omega^*$ is an inner triple derivation, so that $\frak{L}_1(\omega-\omega^*)$ is an inner derivation, proving (vi).
\end{proof}


The following theorem provides some significant infinite dimensional examples of Lie algebras in which every derivation is inner.  Its proof is in the spirit of \cite{PluRus14}.

\begin{theorem}\label{thm:5.7}
Let $V$ be a von Neumann algebra considered as a Jordan triple system with the triple product
$
\{xyz\}=(xy^*z+zy^*x)/2.$  Then every structural transformation on $V$ is an inner structural transformation.  Hence, every derivation of the TKK Lie algebra $\frak{L}(V)$ is inner.
\end{theorem}
\begin{proof}
Let $S$ be a structural transformation on the von Neumann algebra $V$ and to avoid cumbersome notation, denote $S^*$ by $\overline{S}$.  From the defining equations, $\overline{S}(1)=S(1)^*$, and if $S(1)=0$, then $S$ is a Jordan derivation.

For an arbitrary structural transformation $S$, write $S=S_0+S_1$ where $S_0=S-1\bo \overline{S}(1)$ is therefore a Jordan derivation and $S_1=1\bo \overline{S}(1)$ is an inner structural transformation.   By the theorem of Sinclair \cite{Sinclair70}, $S_0$ is a derivation and by the theorems of Kadison and Sakai, \cite{Kadison66,Sakai66}, $S_0$ is an inner derivation, say
$S_0(x)=ax-xa$ for some $a\in V$.   By well known structure of the span of commutators in von Neumann algebras due to Pearcy-Topping,  Halmos, Halpern, Fack-de la Harpe, and others  (see \cite{PluRus14} for the references), $a=z+\sum[c_i,d_i]$, where $c_i,d_i\in V$ and  $z$ belongs to the center of $V$.  It follows that
\[
S_0=2\sum_i c_i\bo d_i^*-2\sum_i d_i\bo c_i^*
\]
and is therefore also an inner structural transformation.  The second statement follows from Proposition~\ref{prop:1015141}.
\end{proof}

We determine the structure of $\frak{L}(V)$ when $V$ is a finite von Neumann algebra in Corollary~\ref{cor:1101141} below.

\section{Examples}\label{sect5}

We conclude the paper with some examples of TKK Lie algebras and
 some Jordan triple cocycles.  Let us first note the following immediate consequences of our
 construction.

\begin{theorem} Let $V$ be a Jordan triple with TKK Lie algebra $(\frak{L}(V), \theta)$.
If the k-th Lie cohomology group $H^k(\frak{L}(V),\frak{L}(M))$ vanishes,
then $H^k(V, M)=\{0\}$ and $H^k_\theta(V, M)=\{0\}$.
\end{theorem}

We have noted  the
one-to-one correspondence between the triple derivations
of a Jordan triple $V$ and the $\theta$-invariant Lie derivations of the TKK Lie algebra
$(\frak L(V), \theta)$, as well as the one-to-one correspondence between the Jordan inner derivations
of $V$ and the Lie inner derivations of $(\frak L(V), \theta)$.

\begin{corollary}
Let  $V$ be a finite
dimensional  Jordan triple with semisimple TKK Lie algebra $\frak L(V)$. Then for any finite dimensional $V$-module $M$,
we have $H^1(V,M) = H^2(V,M) = \{0\}$. In particular, every triple derivation from $V$ to $M$
is inner.
\end{corollary}
\begin{proof}
This follows from Whitehead's lemmas $H^1(\frak L (V),\frak L(M)) = H^2(\frak L (V),\frak L(M)) =\{0\}$.
\end{proof}

In fact, in the above corollary, we have $H^k(\frak L (V),\frak L(M))=\{0\}$ for all $k \geq 3$ if $\frak{L}(M)$ is a nontrivial
irreducible module over $\frak{L}(V)$. We refer to \cite{z} for a converse of this result.

\subsection{Examples of cocycles}
Let $V$ be a Jordan triple with TKK
Lie algebra $(\frak{L}(V), \theta)$. We discuss  examples of Jordan triple cocycles in $Z^k(V,M)$,
where $M$ is a triple $V$-module, and compare them with the Lie cocycles in $Z^k(\frak{L}(V), \frak{L}(M))$.
We have shown in the previous section that the space of Jordan triple derivations is exactly the space of
 $1$-cocycles $Z^1_\theta (V,M)= J^1(Z_\theta^1(\frak{L}(V), \frak{L}(M)))$, where the $\theta$-invariant Lie $1$-cocycles
$Z_\theta^1(\frak{L}(V), \frak{L}(M))$ are exactly the $\theta$-invariant Lie derivations from $\frak L(V)$ to $\frak{L}( M)$. We have also shown that
$B^1_\theta (V,M)= J^1(B_\theta^1(\frak{L}(V), \frak{L}(M)))$ is the space of triple inner derivations on $V$, coming from the $\theta$-invariant Lie inner derivations $B_\theta^1(\frak{L}(V), \frak{L}(M))$.

Examples of triple $2$-cocycles can be constructed from Jordan restrictions of Lie
$2$-cocycles.

\begin{example}\label{example:0922141}
If $\omega\in A^2(V,M)$  is extendable with $\frak{L}_2(\omega)\in
Z^2(\frak{L}(V),\frak{L}(M))$, then $\omega=0$.
\end{example}
\begin{proof} For $x,y,z\in V$,
\begin{eqnarray*}
0&=&d_2\frak{L}_2(\omega)((x,0,0),(y,0,0),(0,0,z))\\
&=&(x,0,0)\cdot(\frak{L}(\omega)((y,0,0),(0,0,z)-(y,0,0)\cdot(\frak{L}(\omega)((x,0,0),(0,0,z)\\
&&
+
(0,0,z)\cdot(\frak{L}(\omega)((x,0,0),(y,0,0)-\frak{L}(\omega)([(x,0,0),(y,0,0],(0,0,z))\\
&&
+\frak{L}(\omega)([(x,0,0),(0,0,z],(y,0,0))-\frak{L}(\omega)([(y,0,0),(0,0,z],(x,0,0))\\
&=&-(0,\omega(x,y)\bo z,0),
\end{eqnarray*}
hence $\omega(x,y)\bo z=0$ for all $x,y,z$ and $\omega=0$.
\end{proof}

\begin{example}
Let $\varphi\in M_0$ be an inner triple derivation, and let $b\in V$.  Define
a linear map $\psi:\frak{L}(V)\rightarrow \frak{L}(M)$ by
\[
\psi(\frak{z}) =[\,[\frak{z}, (0,\varphi,0))]\,, (0,0,b)] \qquad (\frak{z} \in \frak{L}(V)).
\]
Observe that $\psi$ is not $\theta$-invariant. Indeed, it can be seen readily that
$\tilde\theta\psi( x,0,0) = (0, b\bo \vp (x), 0)$ while $\psi(\theta(x,0,0))=0$.
Nevertheless $d_1\psi\in B^2(\frak{L}(V),\frak{L}(M))$ and the triple $2$-coboundary
$Jd_1\psi \in B^2(V,M)$ is given by
\begin{eqnarray*}
Jd_1\psi(x,y)&=&\iota_pd_1\psi((x,0,0), (y,0,0))\\
&=& \iota_p((x,0,0)\cdot \psi(y,0,0)-(y,0,0)\cdot\psi(x,0,0))\\
&=& \iota_p((x,0,0)\cdot (0, -\vp(y) \bo b, 0) - (y,0,0)\cdot(0,-\vp(x)\bo b, 0))\\
&=& \iota_p((\tp{\varphi(y),b,x},0,0) - (\tp{\varphi(x),b,y},0,0))\\
&=& \tp{\varphi(y),b,x}-\tp{\varphi(x),b,y}
\end{eqnarray*}
showing that $B^2(V,M)\ne 0$.
We note that $d_1\psi$ is not $\theta$-invariant
since
\begin{eqnarray*}
&&\widetilde\theta d_1\psi((x,0,0), (y,0,0))= (0,0,\tp{\varphi(y),b,x})- (0,0,\tp{\varphi(x),b,y}),\\
&& d_1\psi ((0,0,x), (0,0,y))= (0,0 - \tp{b, \vp(y),x} + \tp{b, \vp(x),y}).
 \end{eqnarray*}
Also $Jd_1\psi$ need not be extendable. Let $V=M_2(\mathbb{C})$ be the Jordan triple of $2\times 2$
complex matrices. Let $v =\left(\begin{matrix} 1&0\\0&0 \end{matrix}\right)$ and $u = \left(\begin{matrix} 0&0\\0&1 \end{matrix}\right)$. Then we have $u\bo v =0$ and one can find $a,c\in V$ such that
$$Jd_1\psi(u,a)\bo (v+c) + (u+a)\bo Jd_1\psi(v,c)\neq 0.$$ To see this,
let $c=v$. Then it suffices to find $a\in V$ such that
$Jd_1\psi(u,a)\bo 2v \neq 0$, where
$$Jd_1\psi(u,a) = \{\vp(u), b,a\} - \{\vp(a),b,u\}.$$ Let $\vp = m\bo v - v\bo m \in M_0$ where
$m = \left(\begin{matrix} 1&-1\\-1&1 \end{matrix}\right)$. Then we have $\vp(u) = - \{v,m,u\} =
\left(\begin{matrix} 0&1\\1&0 \end{matrix}\right)$. Now let $b=v$. Then we have
$$Jd_1\psi(u,a)\bo v (x) = \left\{\left\{\left(\begin{matrix} 0&1\\1&0 \end{matrix}\right), \left(\begin{matrix} 1&0\\0&0 \end{matrix}\right), a\right\}, v,x\right\}.$$ Finally let $a = v$, then
$$Jd_1\psi(u,a)\bo v (v) =\frac{1}{4}\left(\begin{matrix} 0&1\\1&0 \end{matrix}\right)  \neq 0.$$
\end{example}

We have seen in Example~\ref{example:0922141} that there are no
non-zero extendable elements $\omega\in Z^2(V,M)$ with
$\frak{L}_2(\omega)\in Z^2(\frak{L}(V),\frak{L}(M))$.  The next
example examines this phenomenon for extendable $\omega\in
A^3(V,M)$ with $\frak{L}_3(\omega)\in
Z^3_\theta(\frak{L}(V),\frak{L}(M))$. We state it now as a theorem, in the statement of which, for $a,b\in V$ and $m\in M$,  $[a,b]$ denotes $a\bo b-b\bo a$ and $[m,a]$ denotes $m\bo a-a\bo m$. The proof is provided in subsection~\ref{app:1029143}.

\begin{theorem}\label{prop:0928141}
Let $\omega$ be an extendable element of $A^3(V,M)$.   Then its
Lie extension $\frak{L}_3(\omega)$ is a Lie $3$-cocycle in
$A^3_\theta(\frak{k}(V), \frak{k}(M))$  if and only if  $\omega$
satisfies the following three  conditions:
\begin{equation}
[a,b]\omega(x,y,z)=\omega([a,b]x,y,z)+\omega(x,[a,b]y,z)+\omega(x,y,[a,b]z)
\end{equation}
  for all   $a,b,x,y,z  \in V$;
\begin{equation}
 [\omega(a,b,c),d]=[\omega(d,b,c),a]=[\omega(a,b,d),c]=[\omega(a,d,c),b]
\end{equation}
 for all $a,b,c,d
\in V$; and \begin{equation}
[\omega(x,y,[a,b]z),c]=0.
\end{equation}
for all $x,y,z,a,b,c \in V$.
\end{theorem}

\subsection{Examples of TKK algebras}
  We begin with the following construction  from \cite[Chapter 12]{Meyberg72}, which has its genesis in \cite[pp. 809--810]{KoecherAJM67}. Let $A$ be a unital associative algebra with Lie product the commutator $[x,y]=xy-yx$, Jordan product the anti-commutator $x\circ y=(xy+yx)/2$ and Jordan triple product $\tp{xyz}=(xyz+zyx)/2$ (or $\tp{xyz}=(xy^*z+zy^*x)/2$ if $A$ has an involution). Denote by $Z(A)$ the center of $A$ and by $[A,A]$ the set of finite sums of commutators.

\begin{proposition}\label{prop:1101141} Let $A$ be a unital associative algebra with or without an involution considered as a Jordan triple system.
If $Z(A)\cap [A,A]=\{0\}$, then the mapping $(x,a\bo b,y)\mapsto \matr{ab}{x}{y}{-ba}$ is an isomorphism of the TKK Lie algebra $\frak{L}(A)$  onto the Lie subalgebra
\begin{equation}\label{eq:1103141}
\left\{\matr{u+\sum[v_i,w_i]}{x}{y}{-u+\sum[v_i,w_i]}:u,x,y,v_i,w_i\in A\right\}
\end{equation}
of the Lie algebra $M_2(A)$ with the commutator product.
\end{proposition}

\begin{corollary}\label{cor:1101141}
Let $V$ be a finite von Neumann algebra.  Then
$\frak{L}(V)$ is isomorphic to the Lie algebra $[M_2(V),M_2(V)]$.
\end{corollary}
\begin{proof}
The center valued trace of $V$ is zero on $[V,V]$ and the identity on $Z(V)$, so the theorem applies. Since $M_2(V)$ is also a finite von Neumann algebra, $[M_2(V),M_2(V)]$ coincides with the elements of $M_2(V)$ of central trace zero (by \cite[Theoreme 3.2]{FacdelaHar80}), so it remains to show that every such element has the form (\ref{eq:1103141}).  For this one can use the argument from \cite[pp. 129--130]{Meyberg72} as follows: if $\matr{a}{b}{c}{d}\in M_2(V)$ has central trace zero, then $\tr (a)=-\tr(d)$ and
\[
\matr{a}{b}{c}{d}=\matr{b'+c'}{b}{c}{-b'+c'},
\]
where $c'=(a+d)/2$ and $b'=(a-d)/2$.
\end{proof}

In a properly infinite von Neumann algebra, the assumption $Z(A)\cap [A,A]=\{0\}$ fails since $A=[A,A]$.  This assumption also fails in the Murray-von Neumann algebra of measurable operators affiliated with a factor of type $II_1$ (\cite{Thom}).  For a finite factor of type $I_n$, Corollary~\ref{cor:1101141} states that the classical Lie algebras $sl(2n,\IC)$ of type A  are TKK Lie algebras.  Similarly, the TKK Lie algebra of a Cartan factor of type 3 on an $n$-dimensional Hilbert space is the classical Lie algebra
$sp(2n,\IC)$  of type C (\cite[Theorem 3,p. 131]{Meyberg72}).
More examples
of TKK Lie algebras can be found in \cite[1.4]{chu} and \cite[Chapter III]{Koecher69}.

\section{Proofs of Theorems~\ref{m} and \ref{prop:0928141}}
\subsection{Proof of Theorem~\ref{m}}\label{app:1029141}
\ \\
For the convenience of the reader, we repeat the statement of Theorem~\ref{m}.
  \begin{thm}
   Let $V$ be a Jordan triple and let $\frak L (V)$
 be its TKK Lie algebra. Let $M$ be a triple $V$-module. Then
 $\mathcal{L} (M)$ is a Lie $\frak L(V)$-module.
 \end{thm}
For the proof, we  
  are required to show that
 \begin{eqnarray}\label{eq:0711141}
[(a,h,b),(c,k,d)]\cdot (m,\varphi,n)&=&
(a,h,b)\cdot ((c,k,d)\cdot (m,\varphi,n))\\
&-&(c,k,d)\cdot((a,h,b))\cdot(m,\varphi,n)).\nonumber
\end{eqnarray}
Let $L$ denote the left side of (\ref{eq:0711141}).  Then
\begin{eqnarray}\label{eq:0713144}
L&=&(\underbrace{hc-ka}_A,\underbrace{[h,k]+a\bo d-c\bo b}_H,\underbrace{k^\natural b-h^\natural d}_B)\cdot (m,\varphi,n)\\
\nonumber &=&(\underbrace{Hm-\varphi A}_{L_1},\underbrace{[H,\varphi]+A\bo n-m\bo B}_{L_2},\underbrace{\varphi^\natural B-H^\natural n}_{L_3}).
\end{eqnarray}
We can assume that $h=x\bo y$, $k=u\bo v$ so that
\begin{itemize}
\item $A=\tp{xyc}-\tp{uva}$
\item $H=\tp{xyu}\bo v-u\bo\tp{vxy}+a\bo d-c\bo b$
\item $B=\tp{vub}-\tp{yxd}$.
\end{itemize}

Let $R$ denote the right side of (\ref{eq:0711141}).
Then
\begin{eqnarray}\label{eq:0713145}
R&=&(a,h,b)\cdot (\underbrace{km-\varphi c}_C,\underbrace{[k,\varphi]+c\bo n-m\bo d}_\Phi,\underbrace{\varphi^\natural d-k^\natural n}_D)\\
\nonumber &-&(c,k,d)\cdot(\underbrace{hm-\varphi a}_{C'},\underbrace{[h,\varphi]+a\bo n-m\bo b}_{\Phi'},\underbrace{\varphi^\natural b-h^\natural n}_{D'})\\
\nonumber &=&(\underbrace{hC-\Phi a}_{R_1},\underbrace{[h,\Phi]+a\bo D-C\bo b}_{R_2},\underbrace{\Phi^\natural b-h^\natural D}_{R_3})\\
\nonumber &-&(\underbrace{kC'-\Phi'c}_{R_1'},\underbrace{[k,\Phi']+c\bo D'-C'\bo d}_{R_2'},\underbrace{\Phi'^\natural -k^\natural D'}_{R_3})\\
\nonumber &=&(R_1-R_1',R_2-R_2',R_3-R_3').
\end{eqnarray}
As above, with $h=x\bo y$, $k=u\bo v$ and with $\varphi=w\bo p + q\bo z$, with $p,q\in M$,  we have
\begin{itemize}
\item $C=\tp{uvm}-\tp{wpc}-\tp{qzc}$
\item $D=\tp{pwd}+\tp{zqd}-\tp{vun}$
\item $\Phi=\tp{uvw}\bo p-w\bo\tp{puv}+\tp{uvq}\bo z-q\bo\tp{zuv}+
c\bo n-m\bo d$
\item $C'=\tp{xym}-\tp{wpa}-\tp{qza}$
\item $D'=\tp{pwb}+\tp{zqb}-\tp{yxn}$
\item $\Phi'=\tp{xyw}\bo p-w\bo\tp{pxy}+\tp{xyq}\bo z-q\bo\tp{zxy}+
a\bo n-m\bo b$
\end{itemize}
\smallskip

We now show that $L_1=R_1-R_1'$.
We have from (\ref{eq:0713144})
\begin{eqnarray}\label{eq:0712141}
 L_1&=&Hm-\varphi A\\\nonumber
&=& \tpl{xyu}{v}{m}-\tpc{u}{vxy}{m}+\tp{adm}-\tp{cbm}
  -\tp{wpA}\\\nonumber
&&-\tp{qzA}= \underbrace{\tpl{xyu}{v}{m}}_1-\underbrace{\tpc{u}{vxy}{m}}_1+\underbrace{\tp{adm}}_5-\underbrace{\tp{cbm}}_8
\\
&&-\underbrace{\tpr{w}{p}{xyc}}_2+\underbrace{\tpr{w}{p}{uva}}_6-\underbrace{\tpr{q}{z}{xyc}}_3+\underbrace{\tpr{q}{z}{uva}}_7\nonumber
\end{eqnarray}
and from (\ref{eq:0713145})
 \begin{eqnarray}\label{eq:0712142}
 R_1&=&hC-\Phi a\\\nonumber
 &=&\underbrace{\tpr{x}{y}{uvm}}_1-  \underbrace{\tpr{x}{y}{wpc}}_2-  \underbrace{\tpr{x}{y}{qzc}}_3-  \underbrace{\tp{cna}}_4+  \underbrace{\tp{mda}}_5\\\nonumber
& &-\underbrace{\tpl{uvw}{p}{a}}_6+\underbrace{\tpc{w}{puv}{a}}_6-\underbrace{\tpl{uvq}{z}{a}}_7+\underbrace{\tpc{q}{zuv}{a}}_7\nonumber
\end{eqnarray}
 \smallskip
 and
 \begin{eqnarray}\label{eq:0712143}
R_1'&=&kC'-\Phi' c\\\nonumber
&=&  \underbrace{\tpr{u}{v}{xym}}_1-  \underbrace{\tpr{u}{v}{wpa}}_6-  \underbrace{\tpr{u}{v}{qza}}_7
-  \underbrace{\tpl{xyw}{p}{c}}_2\\
&&+  \underbrace{\tpc{w}{pxy}{c}}_2
-\underbrace{\tpl{xyq}{z}{c}}_3+\underbrace{\tpc{q}{zxy}{c}}_3-\underbrace{\tp{anc}}_4+\underbrace{\tp{mbc}}_8\nonumber
\end{eqnarray}
From (\ref{eq:0712141})-(\ref{eq:0712143}), we have $L_1=R_1-R_1'$.  In
 (\ref{eq:0712141})-(\ref{eq:0712143}) we have indicated which terms cancel. To see that the terms labeled 6 cancel, replace $\tpr{u}{v}{wpa}$ by $\tpl{uvw}{p}{a}-\tpc{w}{vup}{a}+\tpr{w}{p}{uva}$. Similarly, to see that the terms labeled 7 cancel, replace $\tpr{u}{v}{qza}$ by $\tpl{uvq}{z}{a}-\tpc{q}{vuz}{a}+\tpr{q}{z}{uva}$.\smallskip

We next show that $L_2=R_2-R_2'$.  We have from (\ref{eq:0713144})

\begin{eqnarray}\label{eq:0712144}
L_2&=&[H,\varphi]+a\bo n-m\bo B\\
\nonumber &=&[H,w\bo p+q\bo z]+\tp{xyc}\bo n-\tp{uva}\bo n-m\bo\tp{vub}+m\bo\tp{yxd}\\\nonumber
&=&[\tp{xyu}\bo v,w\bo p]-[u\bo \tp{vxy},w\bo p]+[a\bo d,w\bo p]-[c\bo b,w\bo p]\\\nonumber
&&+[\tp{xyu}\bo v,q\bo z]-[u\bo\tp{vxy},q\bo z]+[a\bo d,q\bo z]-[c\bo b,q\bo z]\\\nonumber
&&+\tp{xyc}\bo n-\tp{uva}\bo n-m\bo\tp{vub}+m\bo\tp{yxd}
\end{eqnarray}
and from (\ref{eq:0713145})
\begin{eqnarray}\label{eq:0712145}
R_2&=&[h,\Phi]+a\bo D-C\bo b\\\nonumber
&=&[x\bo y,[u\bo v,w\bo p+q\bo z]]+[x\bo y,c\bo n]-[x\bo y,m\bo d]\\\nonumber
&&+a\bo\tp{pwd}+a\bo\tp{zqd}-a\bo\tp{vun}-\tp{uvm}\bo b\\\nonumber
&&+\tp{wpc}\bo b+\tp{qzc}\bo b\\\nonumber
&=&[x\bo y,\tp{uvw}\bo p]-[x\bo y,w\bo\tp{puv}]+[x\bo y,\tp{uvq}\bo z]\\
\nonumber &&-[x\bo y,q\bo\tp{zuv}]+[x\bo y,c\bo n]-[x\bo y,m\bo d]+a\bo\tp{pwd}\\
\nonumber &&-a\bo \tp{vun}-\tp{uvm}\bo b+\tp{wpc}\bo b+\tp{qzc}\bo b,
\end{eqnarray}

and
\begin{eqnarray}\label{eq:0712146}
R'_2&=&[k,\Phi']+c\bo D'-C'\bo d\\\nonumber
&=&[u\bo v,[x\bo y,w\bo p+q\bo z]]+[u\bo v,a\bo n]-[u\bo v,m\bo b]\\\nonumber
&&+c\bo\tp{pwd}+c\bo\tp{zqb}-c\bo\tp{yxn}-\tp{xym}\bo d\\\nonumber
&&+\tp{wpa}\bo d+\tp{qza}\bo d\\\nonumber
&=&[u\bo v,\tp{xyw}\bo p]-[u\bo v,w\bo\tp{pxy}]+[u\bo v,\tp{xyq}\bo z]\\
\nonumber &&-[u\bo v,q\bo\tp{zxy}]+[u\bo v,a\bo n]-[u\bo v,m\bo b]+c\bo\tp{pwb}\\
\nonumber &&+c\bo\tp{zqb}
-\tp{xym}\bo d+\tp{wpa}\bo d+\tp{qza}\bo d-c\bo\tp{yxn}
\end{eqnarray}

From (\ref{eq:0712144}) we have
\begin{eqnarray}\label{eq:0713141}
L_2&=&\tpl{xyu}{v}{w}\bo p-w\bo\tpc{v}{xyu}{p}-\tpc{u}{vxy}{w}\bo p\\
\nonumber &&+w\bo \tpl{vxy}{u}{p}+\tp{adw}\bo p-w\bo\tp{dap}-\tp{cbw}\bo p\\
\nonumber&&+w\bo\tp{bcp}+\tpl{xyu}{v}{q}\bo z-q\bo \tpc{v}{xyu}{z}-\tpc{u}{vxy}{q}\bo z\\
\nonumber&&+q\bo \tpl{vxy}{u}{z}+\tp{adq}\bo z-q\bo \tp{daz}-\tp{cbq}\bo z+q\bo \tp{bcz}\\
\nonumber&&+\tp{xyc}\bo n-\tp{uva}\bo n-m\bo\tp{vub}+m\bo \tp{yxd}\\
\nonumber&=&(\underbrace{\tpl{xyu}{v}{w}-\tpc{u}{vxy}{w}}_1+\underbrace{\tp{adw}}_{19}-\underbrace{\tp{cbw}}_6)\bo p\\
\nonumber&&+w\bo (\underbrace{-\tpc{v}{xyu}{p}+ \tpl{vxy}{u}{p}}_2\underbrace{-\tp{dap}}_{19}+\underbrace{\tp{bcp}}_6)\\
\nonumber&&+(\underbrace{\tpl{xyu}{v}{q}-\tpc{u}{vxy}{q}}_3+\underbrace{\tp{adq}}_{20}-\underbrace{\tp{cbq}}_7)\bo z\\
\nonumber&&+q\bo(\underbrace{-\tpc{v}{xyu}{z}+ \tpl{vxy}{u}{z}}_4\underbrace{-\tp{daz}}_{20}+\underbrace{\tp{bcz}}_7)\\
\nonumber&&+(\underbrace{\tp{xyc}}_{11}-\underbrace{\tp{uva}}_9)\bo n+m\bo(-\underbrace{\tp{vub}}_{13}+\underbrace{\tp{yxd}}_{12}).
\end{eqnarray}

From (\ref{eq:0712145}) we have
\begin{eqnarray}\label{eq:0713142}
R_2&=&\underbrace{\tpr{x}{y}{uvw}}_1\bo p-\underbrace{\tp{uvw}\bo\tp{yxp}}_{15}-\underbrace{\tp{xyw}\bo \tp{puv}}_{16}\\
\nonumber&&+w\bo\underbrace{\tpr{y}{x}{puv}}_2+\underbrace{\tpr{x}{y}{uvq}}_3\bo z-\underbrace{\tp{uvq}\bo \tp{yxz}}_{17}\\
\nonumber&&-\underbrace{\tp{xyq}\bo\tp{zuv}}_{18}+q\bo \underbrace{\tpr{y}{x}{zuv}}_4+\underbrace{\tp{xyc}}_{11}\bo n-c\bo\underbrace{\tp{yxn}}_{10}\\
\nonumber&&-\underbrace{\tp{xym}}_{14}\bo d+m\bo \underbrace{\tp{yxd}}_{12}+a\bo\underbrace{\tp{pwd}}_{19}+a\bo\underbrace{\tp{zqd}}_{20}\\
\nonumber&&-a\bo \underbrace{\tp{vun}}_8-\underbrace{\tp{uvm}}_5\bo b+\underbrace{\tp{wpc}}_6\bo b+\underbrace{\tp{qzc}}_7\bo b.
\end{eqnarray}

From (\ref{eq:0712146}) we have
\begin{eqnarray}\label{eq:0713143}
R'_2&=&\underbrace{\tpr{u}{v}{xyw}}_1\bo p-\underbrace{\tp{xyw}\bo \tp{vup}}_{16}-
\underbrace{\tp{uvw}\bo\tp{pxy}}_{15}\\
\nonumber&&+w\bo\underbrace{\tpr{v}{u}{pxy}}_2+\underbrace{\tpr{u}{v}{xyq}}_3\bo z-\underbrace{\tp{xyq}\bo \tp{vuz}}_{18}\\
\nonumber&&-\underbrace{\tp{uvq}\bo\tp{yxz}}_{17}+q\bo\underbrace{\tpr{v}{u}{yxz}}_4
+\underbrace{\tp{uva}}_9\bo n-a\bo\underbrace{\tp{vun}}_8\\
\nonumber&&-\underbrace{\tp{uvm}}_5\bo b+m\bo \underbrace{\tp{vub}}_{13}+c\bo\underbrace{\tp{pwb}}_6+c\bo\underbrace{\tp{zqb}}_7\\
\nonumber&&-\underbrace{\tp{xym}}_{14}\bo d+\underbrace{\tp{wpa}}_{19}\bo d+\underbrace{\tp{qza}}_{20}\bo d-c\bo\underbrace{\tp{yxn}}_{10}.
\end{eqnarray}

From (\ref{eq:0712144})-(\ref{eq:0712146}), we have $L_2=R_2-R_2'$.  In
 (\ref{eq:0712144})-(\ref{eq:0712146}) we have indicated which terms cancel.
 The terms labeled 1--4 cancel by the main identity. The terms labeled 5 and 8-18 cancel in pairs.  The terms labeled 6,7,19,20 all cancel because of the following  identity:
 \[
 \tp{abc}\bo d-c\bo \tp{bad}=a\bo \tp{dcb}-\tp{cda}\bo b.
 \]
 which follows from the main identity
 \[
 \tpr{a}{b}{cde}-\tpr{c}{d}{abe}=\tpl{abc}{d}{e}-\tpc{c}{bad}{e}
 \]
 by interchanging $(a,b)$ with $(c,d)$ and noticing that the left side changes sign.

 It remains to show that $L_3=R_3-R_3'$.  We leave this as an exercise for the reader.

\subsection{Proof of Theorem~\ref{prop:0928141}}\label{app:1029143}
\ \\
For the reader's convenience, we repeat the statement of  Theorem~\ref{prop:0928141}, recalling that
we write $[a,b]$ for $a\bo b-b\bo a$ and $[m,a]$ for $m\bo a-a\bo m$ for $a,b \in V$ and $m\in M$.
\begin{thm}
Let $\omega$ be an extendable element of $A^3(V,M)$.   Then its
Lie extension $\frak{L}_3(\omega)$ is a Lie $3$-cocycle in
$A^3_\theta(\frak{k}(V), \frak{k}(M))$  if and only if  $\omega$
satisfies the following three  conditions:
\begin{equation}\label{eq:0930141}
[a,b]\omega(x,y,z)=\omega([a,b]x,y,z)+\omega(x,[a,b]y,z)+\omega(x,y,[a,b]z)
\end{equation}
  for all   $a,b,x,y,z  \in V$;
\begin{equation}\label{eq:0928144}
 [\omega(a,b,c),d]=[\omega(d,b,c),a]=[\omega(a,b,d),c]=[\omega(a,d,c),b]
\end{equation}
 for all $a,b,c,d
\in V$; and \begin{equation}\label{eq:1002148}
[\omega(x,y,[a,b]z),c]=0.
\end{equation}
for all $x,y,z,a,b,c \in V$.
\end{thm}

Let $\omega\in A^3(V,M)$ be extendable and let
$\psi=d_3\frak{L}_3(\omega)$ ($\psi$ is $\theta$-invariant since 3
is odd).   Write $X_j=(x_j,a_j\bo b_j-b_j\bo a_j,x_j)\in
\frak{k}(V)$ as $X_j=(x_j,0,x_j)+(0,[a_j,b_j],0)$.
By the alternating character of $\psi$, it is a Lie 3-cocycle, that is,  $\psi(X_1,X_2,X_3,X_4)=0$ for $X_j\in\frak{k}(V)$, if and only if the following five equations hold for $a_i,b_i,x_i\in V$.
\begin{equation}\label{eq:0923141}
\psi((x_1,0,x_1),(x_2,0,x_2),(x_3,0,x_3),(x_4,0,x_4))=0,\quad\hbox{ (4 variables)}
\end{equation}
\begin{equation}\label{eq:0923142}
\psi((x_1,0,x_1),(x_2,0,x_2),(x_3,0,x_3),(0,[a_4,b_4],0))=0,\quad\hbox{ (5 variables)}
\end{equation}
\begin{equation}\label{eq:0923143}
\psi((x_1,0,x_1),(x_2,0,x_2),(0,[a_3,b_3],0),(0,[a_4,b_4],0))=0,\quad\hbox{ (6 variables)}
\end{equation}
\begin{equation}\label{eq:0923144}
\psi((x_1,0,x_1),(0,[a_2,b_2],0),(0,[a_3,b_3],0),(0,[a_4,b_4],0))=0,\quad\hbox{ (7 variables)}
\end{equation}
\begin{equation}\label{eq:0923145}
\psi((0,[a_1,b_1],0),(0,[a_2,b_2],0),(0,[a_3,b_3],0),(0,[a_4,b_4],0))=0.\quad\hbox{ (8 variables)}
\end{equation}

Note that  (\ref{eq:0930141})-(\ref{eq:1002148}) involve 5,4 and 6 variables respectively, so there is an additional amount of redundancy in
 (\ref{eq:0923141})-(\ref{eq:0923145}).
 We shall begin by showing that (\ref{eq:0923141})-(\ref{eq:0923145}) imply (\ref{eq:0930141})-(\ref{eq:1002148}).

Straightforward calculation of (\ref{eq:0923141}), using (\ref{eq:1015141}) and (\ref{eq:1015142}), shows that it is equivalent to
\begin{equation}\label{eq:0923146}
[x_1,\omega(x_2,x_3,x_4)]-[x_2,\omega(x_1,x_3,x_4)]
\end{equation}
\[
+[x_3,\omega(x_1,x_2,x_4)]-[x_4,\omega(x_1,x_2,x_3)]=0.
\]
We shall see shortly that (\ref{eq:0923146}) is redundant since it  will follow from the  identity (\ref{eq:0928144}), which will be proved using (\ref{eq:0923143}).  However, (\ref{eq:0923146}) will be used later, in the proof that  (\ref{eq:0930141})-(\ref{eq:1002148}) imply  (\ref{eq:0923141})-(\ref{eq:0923145}).

Similarly,  (\ref{eq:0923142}) is equivalent to
\[
-\tp{a_4b_4\omega(x_1,x_2,x_3)}+\tp{b_4a_4\omega(x_1,x_2,x_3)}
\]
\[
+\omega(\tp{a_4b_4x_1},x_2,x_3)-\omega(\tp{b_4a_4x_1},x_2,x_3)
\]
\[
-\omega(\tp{a_4b_4x_2},x_1,x_3)+\omega(\tp{b_4a_4x_2},x_1,x_3)
\]
\[
+\omega(\tp{a_4b_4x_3},x_1,x_2)-\omega(\tp{b_4a_4x_3},x_1,x_2)=0.
\]
which can be rewritten as
\begin{equation}\label{eq:0923147}
[a_4,b_4](\omega(x_1,x_2,x_3))=\omega([a_4,b_4]x_1,x_2,x_3)
\end{equation}
\[
+\omega(x_1,[a_4,b_4]x_2,x_3)+\omega(x_1,x_2,[a_4,b_4]x_3),
\]
proving ({\ref{eq:0930141})  (assuming only  (\ref{eq:0923142})).

An interpretation of (\ref{eq:0923147}) is that the inner triple derivation $[a,b]$ (for the triple product $\tp{\cdot,\cdot,\cdot}$ of $V$)  is also a ``triple derivation'' for the (ad hoc $M$-valued) triple product $(x,y,z)\mapsto \omega(x,y,z)$ of $V$.

In order to proceed efficiently, it is convenient to state the following formulas. First, for
$a_i$ and $b_i$ in $V$, by (\ref{eq:1015142}),
\begin{equation}\label{eq:0926141}
\frak{L}_3(\omega)((0,[a_1,b_1],0),(0,[a_2,b_2],0),(0,[a_3,b_3],0))=(0,\Lambda,0)
\end{equation}
where
\begin{eqnarray}\nonumber\label{eq:0928145}
\Lambda&=&[\omega(a_1,a_2,a_3),b_1+b_2+b_3]-[\omega(b_1,a_2,a_3),a_1+b_2+b_3]\\
&+&[\omega(b_1,b_2,a_3),a_1+a_2+b_3]-[\omega(b_1,b_2,b_3),a_1+a_2+a_3]\\\nonumber
&+&[\omega(b_1,a_2,b_3),a_1+b_2+a_3]+[\omega(a_1,b_2,b_3),b_1+a_2+a_3]\\\nonumber
&-&[\omega(a_1,b_2,a_3),b_1+a_2+b_3]-[\omega(a_1,a_2,b_3),b_1+b_2+a_3].
\end{eqnarray}
Second, for $a,b,c\in V$ and $m\in M$, by (\ref{eq:0418151}),
\begin{equation}\label{eq:0928147}
(0,[a,b],0)\cdot (0,[m,c],0)=(0,[[a,b]m,c]+[m,[a,b]c],0),
\end{equation}
and, for $a_i$ and $b_i$ in $V$, by (\ref{eq:0418152}),
\begin{equation}\label{eq:09281411}
[(0,[a_1,b_1],0),(0,[a_2,b_2],0)]=(0,[[a_1,b_1]a_2,b_2]+[[b_1,a_1]b_2,a_2] ,0).
\end{equation}

Returning to  (\ref{eq:0923143})-(\ref{eq:0923145}) and observing that
\[ \frak{L}_3(\omega) ((*, 0,*), (0,*,0), (*,*,*)) =0,\]
a straightforward calculation of (\ref{eq:0923143}) shows that it is equivalent to
\[
0=-\frak{L}_3(\omega)((0,[x_1,x_2],0),(0,[a_3,b_3],0),(0,[a_4,b_4],0)),
\]
which by (\ref{eq:0926141}) and (\ref{eq:0928145}) is equivalent to
\begin{eqnarray}\nonumber
0&=&[\omega(x_1,a_3,a_4),x_2+b_3+b_4]-[\omega(x_2,a_3,a_4),x_1+b_3+b_4]\\\label{eq:0928142}
&+&[\omega(x_2,b_3,a_4),x_1+a_3+b_4]-[\omega(x_2,b_3,b_4),x_1+a_3+a_4]\\\nonumber
&+&[\omega(x_2,a_3,b_4),x_1+b_3+a_4]+[\omega(x_1,b_3,b_4),x_2+a_3+a_4]\\\nonumber
&-&[\omega(x_1,b_3,a_4),x_2+a_3+b_4]-[\omega(x_1,a_3,b_4),x_2+b_3+a_4].
\end{eqnarray}
We shall now see that (\ref{eq:0928142}) simplifies considerably and gives the same information as
 (\ref{eq:0923144}), namely  (\ref{eq:0928142}) is equivalent to
\begin{equation}\label{eq:0928143}
[\omega(a,b,c),d]=[\omega(d,b,c),a]=[\omega(a,b,d),c]=[\omega(a,d,c),b],
\end{equation}
which is  (\ref{eq:0928144}).
Assuming that this has been done, we will have proved  that (\ref{eq:0923142}) is equivalent to  (\ref{eq:0930141}); and that  (\ref{eq:0923143}),  (\ref{eq:0923144}) and (\ref{eq:0928144}) are equivalent.  We shall  complete the proof by showing that (\ref{eq:0923145}), together with (\ref{eq:0930141}) and  (\ref{eq:0928144}), implies (\ref{eq:1002148}); and then proving that (\ref{eq:0930141})-(\ref{eq:1002148}) imply (\ref{eq:0923141})-(\ref{eq:0923145}).

Note that  (\ref{eq:0930141}), (\ref{eq:1002148}) and the alternating character of $\omega$ imply $[[a,b]\omega(x,y,z),c]=0,$ and that  (\ref{eq:0928144}) and  (\ref{eq:1002148})  imply
\begin{equation}\label{eq:1002149}
[\omega(x,y,z),[a,b]c]=0.
\end{equation}
We continue the proof of Theorem~\ref{prop:0928141} by  showing that (\ref{eq:0928143}) follows from (\ref{eq:0928142}) and that
 (\ref{eq:0923144})   does not contribute any new properties of $\omega$.  After that, we shall deal with (\ref{eq:0923145}).

Since we are assuming (\ref{eq:0923143}), we may set $x_1=0$ and $a_3=0$ in (\ref{eq:0928142}).    The result is
 \[
 [\omega(x_2,b_3,a_4),b_4]-[\omega(x_2,b_3,b_4),a_4]=0.
 \]
 If one repeats this process with ($x_1=0$ and) $a_3=0$ replaced  successivly by $a_4=0,b_3=0,b_4=0$, one obtains three more such equations.
 Next, replace $x_1=0$ by $x_2=0$ to obtain four more such equations.  Finally, setting $a_3=0$ and $a_4=0$ in
 (\ref{eq:0928142}), and repeating with $(a_3,a_4)$ replaced successively with $(a_3,b_4)$, $(b_3,a_4)$, $(b_3,b_4)$
 results in four more such equations.  By changing
 the names of the variables, the resulting twelve equations reduce to  (\ref{eq:0928143})
 (which is (\ref{eq:0928144})).

 We next show that  (\ref{eq:0923144})   yields the same information as
  (\ref{eq:0923143}).
Straightforward calculation of (\ref{eq:0923144}) shows that it is equivalent to
\[
0=(x_1,0,x_1)\cdot\frak{L}_3(\omega)((0,[a_2,b_2],0),(0,[a_3,b_3],0),(0,[a_4,b_4],0)),
\]
which by (\ref{eq:0926141}) equals $(x_1,0,x_1)\cdot(0,\Lambda,0)=(-\Lambda x_1,0,-\Lambda x_1)$ where
\begin{eqnarray}\nonumber\label{eq:0926142}
\Lambda&=&[\omega(a_2,a_3,a_4),b_2+b_3+b_4]-[\omega(b_2,a_3,a_4),a_2+b_3+b_4]\\
&+&[\omega(b_2,b_3,a_4),a_2+a_3+b_4]-[\omega(b_2,b_3,b_4),a_2+a_3+a_4]\\\nonumber
&+&[\omega(b_2,a_3,b_4),a_2+b_3+a_4]+[\omega(a_2,b_3,b_4),b_2+a_3+a_4]\\\nonumber
&-&[\omega(a_2,b_3,a_4),b_2+a_3+b_4]-[\omega(a_2,a_3,b_4),b_2+b_3+a_4],
\end{eqnarray}
Thus,  (\ref{eq:0923144})  results in
\begin{equation}
\Lambda x_1=0.
\end{equation}
where $\Lambda$ is given by (\ref{eq:0926142}).   Comparing this with (\ref{eq:0928142}) shows that
(\ref{eq:0923144}) is equivalent to (\ref{eq:0923143}).

We now have that (\ref{eq:0923143}), (\ref{eq:0923144}), (\ref{eq:0928142}), (\ref{eq:0928143}) and (\ref{eq:0928144}) are equivalent, and that (\ref{eq:0923142}) and  (\ref{eq:0930141}) are equivalent.
It remains, for this part of the proof,  to establish (\ref{eq:1002148}) using (\ref{eq:0923141})-(\ref{eq:0923145}).  This will take some perseverance!

In order to process  (\ref{eq:0923145}) we shall adopt the following self-explanatory notation.
For distinct elements $i,j,k,l\in\{1,2,3,4\}$, set
\begin{equation}\label{eq:09281410}
ijkl_1=(0,[a_i,b_i],0)\cdot \frak{L}_3(\omega)((0,[a_j,b_j],0),(0,[a_k,b_k],0),(0,[a_l,b_l],0))
\end{equation}
and
\begin{equation}\label{eq:09281412}
ijkl_2=\frak{L}_3(\omega)([(0,[a_i,b_i],0),(0,[a_j,b_j],0)],(0,[a_k,b_k],0),(0,[a_l,b_l],0)).
\end{equation}
Then equation  (\ref{eq:0923145}) for $\psi=d_3\frak{L}_3(\omega)$ is restated as:
  \begin{equation}\label{eq:0929141}
 0=1234_1-2134_1+3124_1-4123_1
 \end{equation}
 \[
-1234_2+1324_2-1423_2-2314_2+2413_2-3412_2.
 \]
By   (\ref{eq:09281410}), using (\ref{eq:0926141})-(\ref{eq:0928145}),
\[
ijkl_1=(0,[a_i,b_i],0)\cdot (0,\Lambda_{j,k,l},0)
\]
where
\begin{eqnarray}\nonumber\label{eq:0928146}
\Lambda_{j,k,l}&=&[\omega(a_ j,a_ k,a_l),b_ j+b_ k+b_l]-[\omega(b_ j,a_ k,a_l),a_ j+b_ k+b_l]\\
&+&[\omega(b_ j,b_ k,a_l),a_ j+a_ k+b_l]-[\omega(b_ j,b_ k,b_l),a_ j+a_ k+a_l]\\\nonumber
&+&[\omega(b_ j,a_ k,b_l),a_ j+b_ k+a_l]+[\omega(a_ j,b_ k,b_l),b_ j+a_ k+a_l]\\\nonumber
&-&[\omega(a_ j,b_ k,a_l),b_ j+a_ k+b_l]-[\omega(a_ j,a_ k,b_l),b_ j+b_ k+a_l],
\end{eqnarray}
and by (\ref{eq:0928147}),
\begin{equation}\label{eq:0928149}
ijkl_1=(0,\Gamma_{i,j,k,l},0)
\end{equation}
where
\begin{eqnarray}\nonumber\label{eq:0928148}
\Gamma_{i,j,k,l}&=&[[a_i,b_i]\omega(a_ j,a_ k,a_l),b_ j+b_ k+b_l]+[\omega(a_ j,a_ k,a_l),[a_i,b_i](b_ j+b_ k+b_l)]\\\nonumber
&-&[[a_i,b_i]\omega(b_ j,a_ k,a_l),a_ j+b_ k+b_l]-[\omega(b_ j,a_ k,a_l),[a_i,b_i](a_ j+b_ k+b_l)]\\\nonumber
&+&[[a_i,b_i]\omega(b_ j,b_ k,a_l),a_ j+a_ k+b_l]+[\omega(b_ j,b_ k,a_l),[a_i,b_i](a_ j+a_ k+b_l)]\\
&-&[[a_i,b_i]\omega(b_ j,b_ k,b_l),a_ j+a_ k+a_l]-[\omega(b_ j,b_ k,b_l),[a_i,b_i](a_ j+a_ k+a_l)]\\\nonumber
&+&[[a_i,b_i]\omega(b_ j,a_ k,b_l),a_ j+b_ k+a_l]+[\omega(b_ j,a_ k,b_l),[a_i,b_i](a_ j+b_ k+a_l)]\\\nonumber
&+&[[a_i,b_i]\omega(a_ j,b_ k,b_l),b_ j+a_ k+a_l]+[\omega(a_ j,b_ k,b_l),[a_i,b_i](b_ j+a_ k+a_l)]\\\nonumber
&-&[[a_i,b_i]\omega(a_ j,b_ k,a_l),b_ j+a_ k+b_l]-[\omega(a_ j,b_ k,a_l),[a_i,b_i](b_ j+a_ k+b_l)]\\\nonumber
&-&[[a_i,b_i]\omega(a_ j,a_ k,b_l),b_ j+b_ k+a_l]-[\omega(a_ j,a_ k,b_l),[a_i,b_i](b_ j+b_ k+a_l)]\nonumber.
\end{eqnarray}

By    (\ref{eq:09281412}), using
(\ref{eq:09281411}) and  (\ref{eq:0926141})-(\ref{eq:0928145}),
\begin{eqnarray}\nonumber\label{eq:09281413}
ijkl_2&=&\frak{L}_3(\omega)((0,[[a_i,b_i]a_j,b_j]+[[b_i,a_i]b_j,a_j],0),(0,[a_k,b_k],0),(0,[a_l,b_l],0))\\
&=&(0,\Delta_{i,j,k,l},0),
\end{eqnarray}
where
\begin{eqnarray}\nonumber\label{eq:09281414}
\Delta_{i,j,k,l}&=&[\omega([a_i,b_i]a_ j,a_ k,a_l),b_ j+b_ k+b_l]+[\omega([b_i,a_i]b_ j,a_ k,a_l),a_ j+b_ k+b_l)]\\\nonumber
&-&[\omega(b_ j,a_ k,a_l),[a_i,b_i]a_ j+b_ k+b_l]-[\omega(a_ j,a_ k,a_l),[b_i,a_i](b_ j+b_ k+b_l)]\\\nonumber
&+&[\omega(b_ j,b_ k,a_l),[a_i,b_i]a_ j+a_ k+b_l]+[\omega(a_ j,b_ k,a_l),[b_i,a_i](b_ j+a_ k+b_l)]\\
&-&[\omega(b_ j,b_ k,b_l),[a_i,b_i]a_ j+a_ k+a_l]-[\omega(a_ j,b_ k,b_l),[b_i,a_i](b_ j+a_ k+a_l)]\\\nonumber
&+&[\omega(b_ j,a_ k,b_l),[a_i,b_i]a_ j+b_ k+a_l]+[\omega(a_ j,a_ k,b_l),[b_i,a_i](b_ j+b_ k+a_l)]\\\nonumber
&+&[\omega([a_i,b_i]a_ j,b_ k,b_l),b_ j+a_ k+a_l]+[\omega([b_i,a_i]b_ j,b_ k,b_l),a_ j+a_ k+a_l]\\\nonumber
&-&[\omega([a_i,b_i]a_ j,b_ k,a_l),b_ j+a_ k+b_l]-[\omega([b_i,a_i]b_ j,b_ k,a_l),a_ j+a_ k+b_l]\\\nonumber
&-&[\omega([a_i,b_i]a_ j,a_ k,b_l),b_ j+b_ k+a_l]-[\omega([b_i,a_i]b_ j,a_ k,b_l),a_ j+b_ k+a_l].\nonumber
\end{eqnarray}

We next analyze (\ref{eq:0928148}) and (\ref{eq:09281414}).  First, applying  (\ref{eq:0930141}) to the first bracket on each line of  (\ref{eq:0928148}) and applying (\ref{eq:0928144}) to the expansion of
 those brackets results in 72 terms, 24 of which cancel with all of the terms in the second bracket on each line of  (\ref{eq:0928148}).  Thus the 96 terms in
(\ref{eq:0928148})  are reduced to the 48 terms in

\begin{equation}\label{eq:1001141}
\Gamma_{i,j,k,l}=
\end{equation}
\begin{eqnarray*}\label{eq:0930142}
&&[\omega([a_i,b_i]a_ j,a_ k,a_l),(b_ k+b_l)]
+[\omega(a_ j,[a_i,b_i]a_ k,a_l),(b_ j+b_l)]
+[\omega(a_ j,a_ k,[a_i,b_i]a_l),(b_ j+b_k)]\\\nonumber
&-&[\omega([a_i,b_i]b_ j,a_ k,a_l),(b_ k+b_l)]
-[\omega(b_ j,[a_i,b_i]a_ k,a_l),(a_ j+b_l)]
-[\omega(b_ j,a_ k,[a_i,b_i]a_l),(a_ j+b_k)]\\\nonumber
&+&[\omega([a_i,b_i]b_ j,b_ k,a_l),(a_ k+b_l)]
+[\omega(b_ j,[a_i,b_i]b_ k,a_l),(a_ j+b_l)]\nonumber
+[\omega(b_ j,b_ k,[a_i,b_i]a_l),(a_ j+a_k)]\nonumber\\
&-&[\omega([a_i,b_i]b_ j,b_ k,b_l),(a_ k+a_l)]
-[\omega(b_ j,[a_i,b_i]b_ k,b_l),(a_ j+a_l)]\nonumber
-[\omega(b_ j,b_ k,[a_i,b_i]b_l),(a_ j+a_k)]\\\nonumber
&+&[\omega([a_i,b_i]b_ j,a_ k,b_l),(b_ k+a_l)]\nonumber
+[\omega(b_ j,[a_i,b_i]a_ k,b_l),(a_ j+a_l)]\nonumber
+[\omega(b_ j,a_ k,[a_i,b_i]b_l),(a_ j+b_k)]\\\nonumber
&-&[\omega([a_i,b_i]a_ j,b_ k,b_l),(a_ k+a_l)]\nonumber
-[\omega(a_ j,[a_i,b_i]b_ k,b_l),(b_ j+a_l)]\nonumber
-[\omega(a_ j,b_ k,[a_i,b_i]b_l),(b_ j+a_k)]\\\nonumber
&+&[\omega([a_i,b_i]a_ j,b_ k,a_l),(a_ k+b_l)]\nonumber
+[\omega(a_ j,[a_i,b_i]b_ k,a_l),(b_ j+b_l)]\nonumber
+[\omega(a_ j,b_ k,[a_i,b_i]a_l),(b_ j+a_k)]\\\nonumber
&-&[\omega([a_i,b_i]a_ j,a_ k,b_l),(b_ k+a_l)]\nonumber
-[\omega(a_ j,[a_i,b_i]a_ k,b_l),(b_ j+a_l)]\nonumber
-[\omega(a_ j,a_ k,[a_i,b_i]b_l),(b_ j+b_k)].
\end{eqnarray*}

Second, the 8 first brackets on the lines of  (\ref{eq:09281414}) sum to zero, as can be seen by expanding and noting that the resulting terms cancel in pairs by applying (\ref{eq:0928144}).
Thus  (\ref{eq:09281414})  reduces (initially)  to the sum of the 8 second brackets on the lines of   (\ref{eq:09281414}), namely,
\begin{eqnarray}\nonumber\label{eq:0930143}
\Delta_{i,j,k,l}&=&[\omega([b_i,a_i]b_ j,a_ k,a_l),a_ j+b_ k+b_l)]\\\nonumber
&-&[\omega(a_ j,a_ k,a_l),[b_i,a_i](b_ j+b_ k+b_l)]\\\nonumber
&+&[\omega(a_ j,b_ k,a_l),[b_i,a_i](b_ j+a_ k+b_l)]\\
&-&[\omega(a_ j,b_ k,b_l),[b_i,a_i](b_ j+a_ k+a_l)]\\\nonumber
&+&[\omega(a_ j,a_ k,b_l),[b_i,a_i](b_ j+b_ k+a_l)]\\\nonumber
&+&[\omega([b_i,a_i]b_ j,b_ k,b_l),a_ j+a_ k+a_l]\\\nonumber
&-&[\omega([b_i,a_i]b_ j,b_ k,a_l),a_ j+a_ k+b_l]\\\nonumber
&-&[\omega([b_i,a_i]b_ j,a_ k,b_l),a_ j+b_ k+a_l].\nonumber
\end{eqnarray}
However, there is still more cancellation in (\ref{eq:0930143}) using (\ref{eq:0928144}), and what remains is
\begin{eqnarray}\nonumber\label{eq:0930144}
\Delta_{i,j,k,l}&=&-[\omega(a_ j,a_ k,a_l),[b_i,a_i](b_ k+b_l)]\\\nonumber
&&+[\omega(a_ j,b_ k,a_l),[b_i,a_i](a_ k+b_l)]\\
&&-[\omega(a_ j,b_ k,b_l),[b_i,a_i](a_ k+a_l)]\\\nonumber
&&+[\omega(a_ j,a_ k,b_l),[b_i,a_i](b_ k+a_l)]\\\nonumber.
\end{eqnarray}

The equation  (\ref{eq:0929141}) is thus equivalent to
  \begin{equation}\label{eq:0930145}
 0=\Gamma_{1234}-\Gamma_{2134}+\Gamma_{3124}-\Gamma_{4123}
 \end{equation}
 \[
-\Delta_{1234}+\Delta_{1324}-\Delta_{1423}-\Delta_{2314}+\Delta_{2413}-\Delta_{3412},
 \]
where $\Gamma_{ijkl}$ and $\Delta_{ijkl}$ are given by   (\ref{eq:1001141}) and   (\ref{eq:0930144}).

 We are now going to decompose each term in (\ref{eq:0930145}) into ``irreducible pieces'' as follows. First some notation.   Let $\Sigma$ denote the right side of (\ref{eq:0930145}), let $\Gamma_{ijkl}(a_1=0)$ denote the sum of the terms of $\Gamma_{ijkl}$ which do not involve the variable $a_1$, and
 $\Gamma_{ijkl}(a_1\ne 0)$ the sum of the  terms of
 $\Gamma_{ijkl}$ which contain the variable $a_1$, with similar notation for other variables, for more then one variable, and for $\Delta_{ijkl}$.
With $\Sigma(a_1=0)$ denoting the sum of the terms of $\Sigma$ not containing $a_1$, etc.,
we have ({\it and this is the first of two underlying principles in what follows}) $\Sigma=0$ if and only if $\Sigma(a_1=0)=0$ and $\Sigma(a_1\ne 0)=0$.

We shall use (\ref{eq:1001141}) to process the $\Gamma_{ijkl}$ in (\ref{eq:0930145}) and in parallel use  (\ref{eq:0930144}) to process the $\Delta_{ijkl}$ in (\ref{eq:0930145}).  Here we go!
By  (\ref{eq:1001141}),
 \begin{equation}\label{eq:1001142}
\Gamma_{i,j,k,l}(a_i=0)=0,
\end{equation}
 \begin{equation}\label{eq:1001143}
\Gamma_{i,j,k,l}(a_j=0)=
\end{equation}
\begin{eqnarray*}
&-&[\omega([a_i,b_i]b_ j,a_ k,a_l),(b_ k+b_l)]
-[\omega(b_ j,[a_i,b_i]a_ k,a_l),b_l]
-[\omega(b_ j,a_ k,[a_i,b_i]a_l),b_k]\\\nonumber
&+&[\omega([a_i,b_i]b_ j,b_ k,a_l),(a_ k+b_l)]
+[\omega(b_ j,[a_i,b_i]b_ k,a_l),b_l]\nonumber
+[\omega(b_ j,b_ k,[a_i,b_i]a_l),a_k]\nonumber\\
&-&[\omega([a_i,b_i]b_ j,b_ k,b_l),(a_ k+a_l)]
-[\omega(b_ j,[a_i,b_i]b_ k,b_l),a_l]\nonumber
-[\omega(b_ j,b_ k,[a_i,b_i]b_l),a_k]\\\nonumber
&+&[\omega([a_i,b_i]b_ j,a_ k,b_l),(b_ k+a_l)]\nonumber
+[\omega(b_ j,[a_i,b_i]a_ k,b_l),a_l]\nonumber
+[\omega(b_ j,a_ k,[a_i,b_i]b_l),b_k],\nonumber
\end{eqnarray*}
 \begin{equation}\label{eq:1001144}
\Gamma_{i,j,k,l}(a_k=0)=
\end{equation}
\begin{eqnarray*}
&&[\omega([a_i,b_i]b_ j,b_ k,a_l),b_l]
+[\omega(b_ j,[a_i,b_i]b_ k,a_l),(a_ j+b_l)]\nonumber
+[\omega(b_ j,b_ k,[a_i,b_i]a_l),a_ j]\nonumber\\
&-&[\omega([a_i,b_i]b_ j,b_ k,b_l),a_l]
-[\omega(b_ j,[a_i,b_i]b_ k,b_l),(a_ j+a_l)]\nonumber
-[\omega(b_ j,b_ k,[a_i,b_i]b_l),a_ j]\\\nonumber
&-&[\omega([a_i,b_i]a_ j,b_ k,b_l),a_l]\nonumber
-[\omega(a_ j,[a_i,b_i]b_ k,b_l),(b_ j+a_l)]\nonumber
-[\omega(a_ j,b_ k,[a_i,b_i]b_l),b_ j]\\\nonumber
&+&[\omega([a_i,b_i]a_ j,b_ k,a_l),b_l]\nonumber
+[\omega(a_ j,[a_i,b_i]b_ k,a_l),(b_ j+b_l)]\nonumber
+[\omega(a_ j,b_ k,[a_i,b_i]a_l),b_ j], \\\nonumber
\end{eqnarray*}
and
\begin{equation}\label{eq:1001145}
\Gamma_{i,j,k,l}(a_l=0)=
\end{equation}
\begin{eqnarray*}
&-&[\omega([a_i,b_i]b_ j,b_ k,b_l),a_ k]
-[\omega(b_ j,[a_i,b_i]b_ k,b_l),a_ j]\nonumber
-[\omega(b_ j,b_ k,[a_i,b_i]b_l),(a_ j+a_k)]\\\nonumber
&+&[\omega([a_i,b_i]b_ j,a_ k,b_l),b_ k]\nonumber
+[\omega(b_ j,[a_i,b_i]a_ k,b_l),a_ j]\nonumber
+[\omega(b_ j,a_ k,[a_i,b_i]b_l),(a_ j+b_k)]\\\nonumber
&-&[\omega([a_i,b_i]a_ j,b_ k,b_l),a_ k]\nonumber
-[\omega(a_ j,[a_i,b_i]b_ k,b_l),b_ j]\nonumber
-[\omega(a_ j,b_ k,[a_i,b_i]b_l),(b_ j+a_k)]\\\nonumber
&-&[\omega([a_i,b_i]a_ j,a_ k,b_l),b_ k]\nonumber
-[\omega(a_ j,[a_i,b_i]a_ k,b_l),b_ j]\nonumber
-[\omega(a_ j,a_ k,[a_i,b_i]b_l),(b_ j+b_k)].
\end{eqnarray*}
On the other hand, by (\ref{eq:0930144}),
\begin{equation}\label{eq:1001146}
\Delta_{i,j,k,l}(a_i=0)=0,
\end{equation}
\begin{equation}\label{eq:1001147}
\Delta_{i,j,k,l}(a_j=0)=0,
\end{equation}
\begin{eqnarray}\label{eq:1001148}
\Delta_{i,j,k,l}(a_k=0)&=&[\omega(a_ j,b_ k,a_l),[b_i,a_i]b_l]\\\nonumber
&-&[\omega(a_ j,b_ k,b_l),[b_i,a_i]a_l],
\end{eqnarray}
and
\begin{eqnarray}\label{eq:1001149}
\Delta_{i,j,k,l}(a_l=0)&=&-[\omega(a_ j,b_ k,b_l),[b_i,a_i]a_ k]\\
&&+[\omega(a_ j,a_ k,b_l),[b_i,a_i]b_ k]\nonumber.
\end{eqnarray}
Returning to  (\ref{eq:1001141}), by (\ref{eq:1001142})
\begin{equation}\label{eq:10011410}
\Gamma_{1234}(a_1=0)=0.
\end{equation}
By (\ref{eq:1001143})
 \begin{equation}\label{eq:10011411}
\Gamma_{2134}(a_1=0)=
\end{equation}
\begin{eqnarray*}
&-&[\omega([a_2,b_2]b_ 1,a_ 3,a_4),(b_ 3+b_4)]
-[\omega(b_ 1,[a_2,b_2]a_ 3,a_4),b_4]
-[\omega(b_ 1,a_ 3,[a_2,b_2]a_4),b_3]\\\nonumber
&+&[\omega([a_2,b_2]b_ 1,b_ 3,a_4),(a_ 3+b_4)]
+[\omega(b_ 1,[a_2,b_2]b_ 3,a_4),b_4]\nonumber
+[\omega(b_ 1,b_ 3,[a_2,b_2]a_4),a_3]\nonumber\\
&-&[\omega([a_2,b_2]b_ 1,b_ 3,b_4),(a_ 3+a_4)]
-[\omega(b_ 1,[a_2,b_2]b_ 3,b_4),a_4]\nonumber
-[\omega(b_ 1,b_ 3,[a_2,b_2]b_4),a_3]\\\nonumber
&+&[\omega([a_2,b_2]b_ 1,a_ 3,b_4),(b_ 3+a_4)]\nonumber
+[\omega(b_ 1,[a_2,b_2]a_ 3,b_4),a_4]\nonumber
+[\omega(b_ 1,a_ 3,[a_2,b_2]b_4),b_3], \nonumber
\end{eqnarray*}
 \begin{equation}\label{eq:10011413}
\Gamma_{3124}(a_1=0)=
\end{equation}
\begin{eqnarray*}
&-&[\omega([a_3,b_3]b_ 1,a_ 2,a_4),(b_ 2+b_4)]
-[\omega(b_ 1,[a_3,b_3]a_ 2,a_4),b_4]
-[\omega(b_ 1,a_ 2,[a_3,b_3]a_4),b_2]\\\nonumber
&+&[\omega([a_3,b_3]b_ 1,b_ 2,a_4),(a_ 2+b_4)]
+[\omega(b_ 1,[a_3,b_3]b_ 2,a_4),b_4]\nonumber
+[\omega(b_ 1,b_ 2,[a_3,b_3]a_4),a_2]\nonumber\\
&-&[\omega([a_3,b_3]b_ 1,b_ 2,b_4),(a_ 2+a_4)]
-[\omega(b_ 1,[a_3,b_3]b_ 2,b_4),a_4]\nonumber
-[\omega(b_ 1,b_ 2,[a_3,b_3]b_4),a_2]\\\nonumber
&+&[\omega([a_3,b_3]b_ 1,a_ 2,b_4),(b_ 2+a_4)]\nonumber
+[\omega(b_ 1,[a_3,b_3]a_ 2,b_4),a_4]\nonumber
+[\omega(b_ 1,a_ 2,[a_3,b_3]b_4),b_2]\nonumber,
\end{eqnarray*}
and
 \begin{equation}\label{eq:10011414}
\Gamma_{4123}(a_1=0)=
\end{equation}
\begin{eqnarray*}
&-&[\omega([a_4,b_4]b_ 1,a_ 2,a_3),(b_ 2+b_3)]
-[\omega(b_ 1,[a_4,b_4]a_ 2,a_3),b_3]
-[\omega(b_ 1,a_ 2,[a_4,b_4]a_3),b_2]\\\nonumber
&+&[\omega([a_4,b_4]b_ 1,b_ 2,a_3),(a_ 2+b_3)]
+[\omega(b_ 1,[a_4,b_4]b_ 2,a_3),b_3]\nonumber
+[\omega(b_ 1,b_ 2,[a_4,b_4]a_3),a_2]\nonumber\\
&-&[\omega([a_4,b_4]b_ 1,b_ 2,b_3),(a_ 2+a_3)]
-[\omega(b_ 1,[a_4,b_4]b_ 2,b_3),a_3]\nonumber
-[\omega(b_ 1,b_ 2,[a_4,b_4]b_3),a_2]\\\nonumber
&+&[\omega([a_4,b_4]b_ 1,a_ 2,b_3),(b_ 2+a_3)]\nonumber
+[\omega(b_ 1,[a_4,b_4]a_ 2,b_3),a_3]\nonumber
+[\omega(b_ 1,a_ 2,[a_4,b_4]b_3),b_2].\nonumber
\end{eqnarray*}
On the other hand, by (\ref{eq:1001146})
\begin{equation}\label{eq:10011415}
\Delta_{1234}(a_1=0)=0,\quad \Delta_{1324}(a_1=0)=0,\quad \Delta_{1423}(a_1=0)=0.
\end{equation}
By (\ref{eq:1001148}),
\begin{eqnarray}\label{eq:10011416}
\Delta_{2314}(a_1=0)&=&[\omega(a_ 3,b_ 1,a_4),[b_2,a_2]b_4]\\\nonumber
&-&[\omega(a_ 3,b_ 1,b_4),[b_2,a_2]a_4],
\end{eqnarray}
\begin{eqnarray}\label{eq:10011417}
\Delta_{2413}(a_1=0)&=&[\omega(a_ 4,b_ 1,a_3),[b_2,a_2]b_3]\\\nonumber
&-&[\omega(a_ 4,b_ 1,b_3),[b_2,a_2]a_3],
\end{eqnarray}
and
\begin{eqnarray}\label{eq:10011418}
\Delta_{2413}(a_1=0)&=&[\omega(a_ 4,b_ 1,a_2),[b_3,a_3]b_2]\\\nonumber
&-&[\omega(a_ 4,b_ 1,b_2),[b_3,a_3]a_2].
\end{eqnarray}

By (\ref{eq:0930145}), and  (\ref{eq:10011410})-(\ref{eq:10011418}),
\[
0=\Sigma(a_1=0)=-(\ref{eq:10011411})+(\ref{eq:10011413})-(\ref{eq:10011414})
-(\ref{eq:10011416})+(\ref{eq:10011417})-(\ref{eq:10011418}),
\]
and each of the terms on the right side must be decomposed further.  Here, we are using the notation
(\ref{eq:10011411}) to denote $\Gamma_{2134}(a_1=0)$ and similarly for  (\ref{eq:10011413}), etc.

We shall analyze (\ref{eq:10011411}) first.
By (\ref{eq:10011411}),
 \begin{equation}\label{eq:1002141}
\Gamma_{2134}(a_1=0,a_3=0)=
\end{equation}
\begin{eqnarray*}\label{eq:10011412}\nonumber
&+&[\omega([a_2,b_2]b_ 1,b_ 3,a_4),b_4]
+[\omega(b_ 1,[a_2,b_2]b_ 3,a_4),b_4]\nonumber\\
&-&[\omega([a_2,b_2]b_ 1,b_ 3,b_4),a_4]
-[\omega(b_ 1,[a_2,b_2]b_ 3,b_4),a_4]\nonumber
\end{eqnarray*}
and
 \begin{equation}\label{eq:1002142}
\Gamma_{2134}(a_1=0,a_3\ne 0)=
\end{equation}
\begin{eqnarray*}
&-&[\omega([a_2,b_2]b_ 1,a_ 3,a_4),(b_ 3+b_4)]
-[\omega(b_ 1,[a_2,b_2]a_ 3,a_4),b_4]
-[\omega(b_ 1,a_ 3,[a_2,b_2]a_4),b_3]\\\nonumber
&+&[\omega([a_2,b_2]b_ 1,b_ 3,a_4),a_ 3]
+[\omega(b_ 1,b_ 3,[a_2,b_2]a_4),a_3]\nonumber\\
&-&[\omega([a_2,b_2]b_ 1,b_ 3,b_4),a_ 3]
-[\omega(b_ 1,b_ 3,[a_2,b_2]b_4),a_3]\\\nonumber
&+&[\omega([a_2,b_2]b_ 1,a_ 3,b_4),(b_ 3+a_4)]\nonumber
+[\omega(b_ 1,[a_2,b_2]a_ 3,b_4),a_4]\nonumber
+[\omega(b_ 1,a_ 3,[a_2,b_2]b_4),b_3].\nonumber
\end{eqnarray*}
The identity given by (\ref{eq:1002141}) is ``irreducible'' in the sense that if any of its variables is zero, then it vanishes identically ({\it This is the second of the two underlying principles mentioned earlier}).  However, since it is a consequence of (\ref{eq:0928144}), it does not give any new identities and can be ignored.  We proceed to decompose (\ref{eq:1002142}) as follows.
 \begin{equation}\label{eq:1002143}
\Gamma_{2134}(a_1=0,a_3\ne 0,a_4=0)=
\end{equation}
\begin{eqnarray*}
&-&[\omega([a_2,b_2]b_ 1,b_ 3,b_4),a_ 3]
-[\omega(b_ 1,b_ 3,[a_2,b_2]b_4),a_3]\\\nonumber
&+&[\omega([a_2,b_2]b_ 1,a_ 3,b_4),b_ 3]\nonumber
+[\omega(b_ 1,a_ 3,[a_2,b_2]b_4),b_3].\nonumber
\end{eqnarray*}
 \begin{equation}\label{eq:1002144}
\Gamma_{2134}(a_1=0,a_3\ne 0,a_4\ne 0)=
\end{equation}
\begin{eqnarray*}
&-&[\omega([a_2,b_2]b_ 1,a_ 3,a_4),(b_ 3+b_4)]
-[\omega(b_ 1,[a_2,b_2]a_ 3,a_4),b_4]
-[\omega(b_ 1,a_ 3,[a_2,b_2]a_4),b_3]\\\nonumber
&+&[\omega([a_2,b_2]b_ 1,b_ 3,a_4),a_ 3]
+[\omega(b_ 1,b_ 3,[a_2,b_2]a_4),a_3]\nonumber\\
&+&[\omega([a_2,b_2]b_ 1,a_ 3,b_4),a_4]\nonumber
+[\omega(b_ 1,[a_2,b_2]a_ 3,b_4),a_4]\nonumber
\end{eqnarray*}
The identity given by (\ref{eq:1002143}) is irreducible and can also be ignored, so we proceed to decompose (\ref{eq:1002144}) as follows.
 \begin{equation}\label{eq:1002145}
\Gamma_{2134}(a_1=0,a_3\ne 0,a_4\ne 0,b_3=0)=
\end{equation}
\begin{eqnarray*}
&-&[\omega([a_2,b_2]b_ 1,a_ 3,a_4),b_4]
-[\omega(b_ 1,[a_2,b_2]a_ 3,a_4),b_4]
\\
&+&[\omega(b_ 1,b_ 3,[a_2,b_2]a_4),a_3]\nonumber\\
&+&[\omega([a_2,b_2]b_ 1,a_ 3,b_4),a_4]\nonumber
+[\omega(b_ 1,[a_2,b_2]a_ 3,b_4),a_4],\nonumber
\end{eqnarray*}
and
\begin{equation}\label{eq:1002146}
\Gamma_{2134}(a_1=0,a_3\ne 0,a_4\ne 0,b_3\ne 0)=
\end{equation}
\begin{eqnarray*}
&-&[\omega([a_2,b_2]b_ 1,a_ 3,a_4),b_ 3]
\\
&+&[\omega([a_2,b_2]b_ 1,b_ 3,a_4),a_ 3]
+[\omega(b_ 1,b_ 3,[a_2,b_2]a_4),a_3].\nonumber
\end{eqnarray*}
By using (\ref{eq:0928144}), each of (\ref{eq:1002145}) and (\ref{eq:1002146}) gives the  new identity
\begin{equation}\label{eq:1002147}
[\omega(b_ 1,b_ 3,[a_2,b_2]a_4),a_3]=0,
\end{equation}
which establishes (\ref{eq:1002148}), and at the same time shows that (\ref{eq:10011416}), (\ref{eq:10011417}) and (\ref{eq:10011418}) produce no new identities.

This completes the analysis of (\ref{eq:10011411}), which has produced  (\ref{eq:1002148}).
Since  (\ref{eq:10011413})  is obtained from  (\ref{eq:10011411}) by interchanging the indices 3 and 2, no new information is provided by  (\ref{eq:10011413}).  Similarly, since  (\ref{eq:10011414})  is obtained from  (\ref{eq:10011413}) by interchanging the indices 3 and 4, no new information is provided by  (\ref{eq:10011414}).
 Thus we have found all irreducible expressions which sum to $\Sigma(a_1=0)$, resulting in only one identity, namely  (\ref{eq:1002148}).
This completes the proof  that
(\ref{eq:0923141})-(\ref{eq:0923145}) imply
  (\ref{eq:0930141})-(\ref{eq:1002148}).
(See the paragraph following (\ref{eq:0930145}).)

It is now a simple matter to prove that, conversely,   (\ref{eq:0930141})-(\ref{eq:1002148}) imply
(\ref{eq:0923141})-(\ref{eq:0923145}).
Note that by (\ref{eq:1002148}),
(\ref{eq:1002149}), and (\ref{eq:0928148}),(\ref{eq:09281414}), $\Gamma_{ijkl}$ and $\Delta_{ijkl}$ vanish, showing that   $\Sigma(a_1\ne 0)=0$,
hence   (\ref{eq:0930141})-(\ref{eq:1002148}) imply
(\ref{eq:0923145}).
Since earlier arguments  have shown that \begin{itemize}
\item (\ref{eq:0928144}) $\Rightarrow$ (\ref{eq:0923146}) $\Leftrightarrow$ (\ref{eq:0923141}),
\item
(\ref{eq:0930141}) = (\ref{eq:0923147}) $\Leftrightarrow$ (\ref{eq:0923142}),
\item (\ref{eq:0928144}) $\Leftrightarrow$ (\ref{eq:0928142}) $\Leftrightarrow$ (\ref{eq:0923143})$\Leftrightarrow$ (\ref{eq:0923144}),
\end{itemize}
this completes the proof that  (\ref{eq:0930141})-(\ref{eq:1002148}) imply
(\ref{eq:0923141})-(\ref{eq:0923145}),  and hence the proof of Theorem~\ref{prop:0928141}.

\bibliographystyle{amsalpha}

\end{document}